\documentclass{amsart}

\usepackage{amssymb,amsfonts}
\usepackage[all,arc]{xy}
\usepackage{enumerate}
\usepackage{mathrsfs}
\usepackage{mathtools}
\usepackage{stmaryrd}
\usepackage{url}
\usepackage[colorinlistoftodos]{todonotes}
\usepackage{etoolbox}
\usepackage{graphicx}
\usepackage{placeins}
\usepackage{needspace}
\usepackage{lscape}
\usepackage{pdflscape}

\usepackage[margin=20truemm]{geometry}

\usepackage{spectralsequences}
\usepackage{rotating}

\usepackage{comment}

\usepackage[unicode, psdextra]{hyperref}

\newcommand{\fullpagegraphic}[2][]{%
  \ifstrempty{#1}{%
    \includegraphics[width=\textwidth,height=0.9\textheight,keepaspectratio]{#2}%
  }{%
    \includegraphics[width=\textwidth,height=0.9\textheight,keepaspectratio,#1]{#2}%
  }%
}
\newcommand{\verticalgraphic}[2][]{%
  \ifstrempty{#1}{%
    \includegraphics[width=0.95\textheight, angle=90]{#2}%
  }{%
    \includegraphics[width=0.95\textheight, angle=90,#1]{#2}%
  }%
}

\usepackage[hypcap=false]{caption}

\newcommand{\widegraphic}[2][]{%
  \makebox[\linewidth][c]{%
    \includegraphics[width=1.08\linewidth,keepaspectratio,#1]{#2}%
  }%
}

\usepackage{mathtools}
\DeclarePairedDelimiter{\abs}{\lvert}{\rvert}
\DeclarePairedDelimiter{\toda}{\langle}{\rangle}

\newcommand{\cal}[1]{
  \renewcommand*{\do}[1]{
    \expandafter\newcommand\csname c##1\endcsname{\ensuremath{\mathcal{##1}}}
  }
  \docsvlist{#1}
}
\cal{A, B, C, D, E, F, G, H, I, J, K, L, M, N, O, P, Q, R, S, T, U, V, W, X, Y, Z}

\newcommand{\cat}[1]{
  \renewcommand*{\do}[1]{
    \expandafter\newcommand\csname ##1\endcsname{\ensuremath{\mathrm{##1}}}
  }
  \docsvlist{#1}
}
\cat{TJF, TMF, tmf, Sp, MF, KO, KU, String, JF, MU, Syn, can, Tate, Orb, Mfd, Top, Ab, SpSch, SpStk, CAlg, fin, Lie, Gpd, Stk}

\newcommand{\DeclareMyOperator}[1]{%
  \expandafter\DeclareMathOperator\csname #1\endcsname{#1}
}
\newcommand{\DeclareMathOperators}{\forcsvlist{\DeclareMyOperator}}

\DeclareMathOperators{Fun, Map, Ext, Tot, Spec, QCoh, Ell, GL, SL, gl, Pic, Proj, Hom, tr, Shv, Spf}

\DeclareMathOperator*{\colim}{colim}

\newcommand{\bb}[1]{
  \renewcommand*{\do}[1]{
    \expandafter\newcommand\csname b##1\endcsname{\ensuremath{\mathbb{##1}}}
  }
  \docsvlist{#1}
}
\bb{A, B, C, D, E, F, G, H, I, J, K, L, M, N, O, P, Q, R, S, T, U, V, W, X, Y, Z}

\newcommand{\opcat}[1]{{#1}^{\mathrm{op}}}
\newcommand{\abcat}[1]{{#1}^{\mathrm{ab}}}

\newcommand{\syn}[1]{{#1}^{\mathrm{syn}}}
\newcommand{\conn}[1]{\tau_{\geq 0} {#1}}

\newcommand{\cEll}{\cE ll}
\newcommand{\orM}{\cM_{\mathrm{ell}}^\mathrm{or}}
\newcommand{\ellM}{\cM_{\mathrm{ell}}}
\newcommand{\bkappa}{\overline{\kappa}}
\newcommand{\locTop}{\mathrm{Top}_{\mathrm{loc}}^{\mathrm{CAlg}}}
\newcommand{\Einf}{\bE_\infty}
 
\newcommand{\qF}[1]{\widehat{\mathbb{G}}^Q_{#1}}

\newcommand{\triplerightarrows}{%
  \mathrel{\substack{\textstyle\rightarrow\\[-0.6ex]
                     \textstyle\rightarrow\\[-0.6ex]
                     \textstyle\rightarrow}}%
}

\newtheorem{thm}{Theorem}[section]

\newtheorem{prop}[thm]{Proposition}
\newtheorem{lem}[thm]{Lemma}

\theoremstyle{definition}
\newtheorem{defn}[thm]{Definition}

\theoremstyle{remark}
\newtheorem{rem}[thm]{Remark}

\title{Computation of Topological Jacobi Forms}
\author{Akira Tominaga}
\date{}

\begin{document}

\begin{abstract}
We compute, at the prime $2$, the entire descent spectral sequence converging to the homotopy groups of the spectra of topological Jacobi forms $\TJF_m$ for every index $m\ge1$.  An explicit $\TMF$–cellular decomposition
\(
\TJF_m \simeq \TMF \otimes P_m
\)
reduces the problem to analyzing a finite complex $P_m$ with one even cell in each dimension $\le 2m$ except $2$. We identify all differentials using the cell structure.
\end{abstract}

\maketitle

\section{Introduction}
The spectrum of topological modular forms ($\TMF$) was constructed by Hopkins and Mahowald as a height-$2$ spectrum equipped with a structure map $\pi_{2*}\TMF \to \MF_*$. Through the work of Ando–Hopkins–Strickland \cite{AHS2001} and Ando–Hopkins–Rezk \cite{AndoHopkinsRezk2010}, $\TMF$ arises as the target of the Witten genus
\[
M \String \to \TMF,
\] 
and it has accordingly become a central object in homotopy theory.
Building on Lurie's extensive work on spectral algebraic geometry and elliptic cohomology (\cite{LurieSAG}, \cite{Lurie2009}, \cite{LurieElliptic1}, \cite{LurieElliptic2}, \cite{LurieElliptic3}), Gepner–Meier constructed a $G$-equivariant refinement of $\TMF$ for each compact Lie group $G$ \cite{gepner2023equivariant}. They produced a spectral Deligne–Mumford stack $\cM_G$ and a colimit-preserving functor
\[
\cEll_{G} \colon \opcat{\Sp}_G \to \QCoh (\cM_G)
\]
satisfying
\begin{enumerate}
    \item $\cM_e = \orM$, where $e$ is the trivial group and $\orM$ is the spectral DM stack of oriented elliptic curves constructed in \cite{LurieElliptic2},
    \item $\cM_{\bT} \simeq \cE$, where $\bT$ is the circle and $\cE$ is the universal oriented elliptic curve over $\orM$ (see also \cite[Chapter 4]{meier2022topologicalmodularformslevel}),
	    \item each $\cM_G$ is a relative scheme over $\orM,$
	\end{enumerate}
and the assignment $G \mapsto \cEll_G$ is natural, so $G$-equivariant $\TMF$ can be interpreted as a globally equivariant cohomology theory \cite{gepner2024global2ringsgenuinerefinements}.

Building on this framework, we introduce a family of spectra—the \emph{topological Jacobi forms}
\[
  \TJF_m \coloneqq
  \Gamma \bigl(\cE,\mathcal O_\cE(m e)\bigr) \qquad(m\ge 1)
\]
which arise as the global sections of line bundles on the universal oriented elliptic
curve $\cE \to \orM$.
These spectra package the $RO(\bT)$-graded $\TMF$ and provide a homotopy theoteric refinement of the ring of integral weak Jacobi forms of weight $k$ and index $\frac{m}{2}$
\[ 
\JF_{k,\frac{m}{2}} = \Gamma \bigl(E, p^{*}\omega^{\otimes k+m}\otimes \cO_{E}(m e)\bigr)
\]
defined in, for example, \cite{EichlerZagier1985} and \cite{Kramer1995}.

The present paper gives a complete $2$-local analysis of $\TJF_m$ for every $m\ge1$.
Our main contribution is an explicit computation of the homotopy groups
$\pi_* \TJF_m$.
Although classical Jacobi forms are already subtle, the topological refinement reveals a surprising relationship with complex projective space through the transfer map in equivariant homotopy theory.

\subsection{Main results}

\begin{thm}
\label{thm:intro-main}
For each integer $m\ge1$, there exists an equivalence of $\TMF$-modules
\[
  \TJF_m  \simeq \TMF \otimes P_m,
\]
where $P_m$ is obtained from the stunted complex projective space
$\Sigma^2\bC P^{m-1}_{-1}$ by deleting its $2$-cell (Definition \ref{def:Pm}).
In particular, $\TJF_m$ has exactly one $\TMF$-cell in each even dimension $2d \le 2m$,
$d\ne1$, and the stable attaching maps are detected by the Hopf fibrations
$\eta$ $\nu$, and $2\nu$.
\end{thm}

\begin{thm}
\label{thm:intro-dss}
Let $E_r^{s,t}(m)$ denote the descent spectral sequence converging to
$\pi_{t-s}\TJF_m$.
\begin{enumerate}
  \item The $E_2$-term is
        \[
          E_2^{s,2t}(m) \cong 
          H^{s}\bigl( E; p^* \omega^{\otimes t}\otimes\mathcal O_E(m e)\bigr)
        \]
        and is explicitly described by the cohomology of a Hopf algebroid
        $(B_m,\Sigma_m)$
        obtained from the $GL_2(\bZ / 3)$-Galois cover $E' \to E$ of the universal elliptic curve $E$ (Section \ref{sec:E2term}).
  \item The descent spectral sequence has a horizontal vanishing line at
        $s=24$ on the $E_{24}$-page (Corollary \ref{lem:vanishing}).
  \item The differentials are completely determined for every $m$
        (Section \ref{sec:TJF2} to \ref{sec:TJF7}), from the cell structure, the vanishing line, the (synthetic) Leibniz rule.
\end{enumerate}
\end{thm}

\subsection{Outline of the argument}
Our computation proceeds in four stages.

\begin{enumerate}
\item
  \textbf{Setting up spectral sequences.}
  We work over the \'{e}tale cover
  $\cM_{1}(3) \to \cM_{\mathrm{ell}}$, whose affine coordinate ring is $A=\bZ_{(2)}[a_1,a_3,\Delta^{\pm}]$.
  Pulling back the universal elliptic curve $E$ to $\cM_{1}(3)$ yields a Weierstra{\ss} curve, from which we form the Hopf algebroid
  $(B_m,\Sigma_m)$ controlling $\TJF_m$.
\item
  \textbf{$\TJF_\infty$-case.} Under the change-of-rings theorem \cite{hovey_invertible_1999} \cite{hovey2001moritatheoryhopfalgebroids}, we can work out the $E_2$-page of the DSS for $\TJF_\infty$. It turns out that there is essentially only one $d_3$-differential given by the relation $\eta^4 = 0$, yielding the simple structure of $\pi_* \TJF_\infty$.
\item
  \textbf{Transfer and cellular structure.}
  Exploiting the transfer sequence
  $\Sigma \TMF \to \TMF_\bT \to \TMF$ studied in \cite{gepner2023equivariant},
  we identify $\TJF_m$ as the cofiber of a map
  $\TMF \otimes \Sigma\bC P^{m-1}_{+}\to \TMF \oplus\Sigma\TMF$.
  A comparison with
  the stunted projective spaces \cite{MillerBernoulli} \cite{Mosher}
  reveals the cell structure of the cofiber and proves Theorem \ref{thm:intro-main}.

\item
  \textbf{Higher differentials.}
  The cell decomposition and analysis of $\TJF_\infty$ yields the $E_2$-page and sharply restricts the possible $d_r$ in the descent spectral sequence.
  For $m=2$, we compute all differentials from the fact that $\TJF_2 \simeq \TMF \otimes C\nu.$
  The method extends, with increasing bookkeeping, to $m=3, \dots, 7$.
\end{enumerate}

\subsection{Organization of the paper}
Section \ref{sec:prelim} is about the preliminary on equivariant $\TMF$. Section \ref{sec:setup} and \ref{sec:E2term} recalls circle-equivariant elliptic cohomology, sets notation, and constructs the Hopf algebroid $(B_m,\Sigma_m)$.
Section \ref{sec:cellstr} proves Theorem \ref{thm:intro-main} by comparing with stunted projective spaces.
Sections \ref{sec:TJF2}–\ref{sec:TJF7} compute the $E_2$-term and differentials of descent spectral sequences
for $m=2, \dots ,7$.

\subsection{Conventions}
\begin{enumerate}
    \item $\Sp$ denotes the $\infty$-category of spectra and $\Sp_G$ the category of genuine $G$-spectra; a superscript ${\fin}$ indicates the full subcategory generated by compact objects. We write $\cS$ for the $\infty$-category of spaces and $\cS^{\fin}$ for compact spaces.
    \item $\Sigma^\infty_+ \colon \cS \to \Sp$ and $\Sigma^\infty \colon \cS_* \to \Sp$ are the suspension spectrum functors; we freely regard a pointed space as its suspension spectrum and write $(-)_+$ for adjoining a disjoint basepoint.
    \item $\bS$ denotes the sphere spectrum and $\pi_* X$ the stable homotopy groups. Unless stated otherwise, spectra and spectral stacks are implicitly $2$-local.
    \item $\eta \in \pi_1 \bS$ and $\nu \in \pi_3 \bS$ are the elements in $\pi_* \bS$ represented by Hopf fibrations $S^3 \to S^2$ and $S^7 \to S^4$, respectively.
    \item $\bT = U(1)$ is the circle group and $\rho \colon \bT \hookrightarrow \bC^{\times}$ its tautological complex representation; $S^{m\rho}$ is the associated representation sphere.
    \item $\tmf$ is the connective spectrum of topological modular forms and $\TMF = \tmf[(\Delta^{8})^{-1}]$ its periodic version. We use the same names for elements of $\pi_* \tmf$ as in \cite{Bauer_2008}.
    \item $\orM$ is the spectral Deligne-Mumford stack of oriented elliptic curves \cite{LurieElliptic2}; its underlying $1$-stack of smooth elliptic curves is $\cM_{\mathrm{ell}}$.
    \item $\cM_1^{\mathrm{or}}(n)$ denotes the spectral moduli of oriented elliptic curves with $\Gamma_1(n)$-level structure; the global sections of its structure sheaf give $\TMF_1(n)$.
    \item We use calligraphic letters such as $\cM$ for spectral stacks and ordinary letters $M$ for their underlying classical stacks.
    \item $p \colon \cE \to \orM$ is the universal oriented elliptic curve with identity section $e$. We write $\cO_\cE(me)$ for the line bundle associated to $me$ and $\omega$ for the Hodge line bundle on $\cM_{\mathrm{ell}}$; $p^* \omega$ is its pullback to the universal elliptic curve $E$.
    \item $\CAlg$ denote the category of $\bE_\infty$-rings and $\SpSch$ denote the category of spectral schemes.
\end{enumerate}

\subsection {Acknowledgements}
I would like to express my sincere gratitude to my supervisor, David Gepner, for invaluable guidance, continuous engagement, and insightful discussion throughout my PhD program. I am also grateful to Lennart Meier, Tilman Bauer, Mayuko Yamashita, Agn\`es Beaudry, and Ishan Levy for helpful discussions and constructible comments that significantly improved this research. I thank MPIM for their hospitality where a part of this thesis was written.

\bigskip

\section{Preliminary}\label{sec:prelim}
We briefly recall below the construction of equivariant TMF due to \cite{gepner2023equivariant}. We make no claim of originality; all statements are explained in the cited references. Let $S = \Spec R$ be the spectral scheme associated with an even periodic $\Einf$-ring $R \in \CAlg$.
\begin{defn}
    A spectral elliptic curve $p \colon C \to S$ is a strict abelian group object in $\SpSch_{S}$ such that
    \begin{itemize}
        \item $p$ is flat,
        \item $\conn{p}$ is proper and locally almost of finite presentation, and
        \item for any point $i \colon\Spec k \to \conn{S}$ with $k$ being an algebraically closed field, the pullback $i^* \cC \to \Spec k$ is a (classical) elliptic curve over $k$.
    \end{itemize}
\end{defn}
Completion along the identity section $e \colon S \to C$ defines a formal group $\widehat{C}$ over $S$. When the base ring $R$ is even periodic, the first Chern class induces a formal group $\Spf(R^0 (\bC P^\infty))$ over $R^0 (\bC P^\infty)$; Lurie also constructed the spectral formal group $\qF{R} \coloneqq \Spf(C^* (\bC P^\infty ; R))$ associated to the even periodic ring $R$. These two formal groups are related via \textit{preorientations} and \textit{orientations}.
\begin{defn}
    A preorientation of a strict abelian group object $C$ in $\SpSch_S$ is a map
    \[
    i \colon S^2 \to \Map_{\SpSch_S}(S, C).
    \]
\end{defn}
It turns out that the datum of a preorientation is equivalent to giving a map of strict abelian group objects $\bC P^\infty \to \Map_{\SpSch_S}(S, C)$ in $\cS$. From this, it follows that a preorientation induces a map of formal groups $\qF{R} \to \widehat{C}$ \cite[Proposition 4.3.21]{LurieElliptic2}. 

\begin{defn}
    A preorientation is called an orientation if the corresponding map
    \[
     \qF{R} \to \widehat{C}
    \]
    of formal groups is an equivalence.
\end{defn}
Note that a preorientation is an orientation if and only if the induced map $\omega_{\widehat{C}} \to \Sigma^{-2} R$ of line bundles on $S$ is an equivalence, where $\omega_{\widehat{C}}$ is the dualizing line of $\widehat{C}$ at the identity section \cite[Proposition 4.3.23]{LurieElliptic2}.

For an oriented abelian scheme $C$, one can associate a global equivariant cohomology theory. To explain this, we first recall the notion of orbispaces. Denote by $\Top$ the $1$-category of topological spaces and by $\Gpd$ the $1$-category of groupoids.

\defn A topological $1$-stack is a functor \footnote{Technically, to avoid size issues, we restrict the domain to a small full subcategory $\Top' \subset \Top$ that contains all homeomorphism types of CW complexes and is closed under finite products and subspaces. An example is the full subcategory of $\Top$ whose objects are subspaces of finite products of $\bR^{\bN}$.} $X \colon \opcat{\Top} \to \Gpd$ satisfying the sheaf condition: i.e., for any topological space $T$ and any open cover $\{ U_i \to T \}_{i \in I}$, the canonical map
\[
X(T) \to \lim ( \prod_{i \in I} X(U_i) \rightrightarrows \prod_{i,j \in I} X(U_i \cap U_j) \triplerightarrows \prod_{i,j,k \in I} X(U_i \cap U_j \cap U_k))
\]
is an equivalence of groupoids.

Let $G$ be a compact Lie group. When a space $X$ admits a $G$-action, we can construct a groupoid object
\[
\cdots \triplerightarrows G \times G \times X \rightrightarrows G \times X \to  X
\]
in the ($1$-)category of topological spaces. Sheafifying the functor $M \mapsto \Map_\Top (M , X_*)$ corepresented by this groupoid object yields the associated topological stack $[X / G]$. In particular, when $X$ is a point, one obtains the stack $\bB G$ classifying principal $G$-bundles.

Topological stacks form an $\infty$-category $\Stk_\infty$: objects are topological stacks, 1-morphisms are maps of topological stacks, 2-morphisms are homotopies of maps as prestacks, and so on. For a precise treatment, see \cite[Appendix A]{gepner2023equivariant}. The orbit category is the following full subcategory of $\Stk_\infty$.
\defn The orbit category $\Orb$ is the full subcategory of topological stacks $\Stk_\infty$ spanned by objects of the form $\bB G$, where $G$ is a compact Lie group. It can be described as a $\Top$-enriched category as follows:
\begin{itemize}
    \item objects are topological stacks of the form $\bB G$, and
    \item mapping space is defined by $\Map_{\Orb} (\bB G, \bB H) =\Map_{\Lie} (H, G) _{hG} $ where the $G$-action on the latter is given by conjugation.
\end{itemize}

\defn The presheaf category $\cP(\Orb)$ is called the category of orbispaces. 

Let $\Orb_G$ denote the orbit category of $G$: the full subcategory of $G$-spaces of the form $G/H$ for a closed subgroup $H \triangleleft G$. Recall that, by Elmendorf's theorem, for each compact Lie group $G$, the category of $G$-spaces $\cS^G$ is equivalent to the presheaf category $\cP (\Orb_G)$ on orbits of $G$. 

\begin{prop}[\cite{Linskens_2025}]
There exists a fully faithful embedding $\Orb_G \to \Orb_{/ \bB G}$ given by $G/H \mapsto (\bB H \to \bB G)$.
\end{prop}

This embedding $\Orb_G \to \Orb$ induces a functor $\cS^G \simeq \cP(\Orb_G) \to \cP(\Orb)_{/ \bB G}$ by the left Kan extension.
\begin{prop}\label{prop:Orbembedding} Let $G$ be a compact Lie group. Then the functor
\[
\cS^G \to \cP(\Orb)_{ / \bB G}
\]
is colimit-preserving and fully faithful.
\end{prop}
Take the full subcategory $\abcat{\Orb} \subset \Orb$ spanned by objects equivalent to $\bB G$ for an abelian compact Lie group $G$. Given an oriented elliptic curve $C \to S$, we can construct a functor from $\abcat{\Orb}$ to the category of ringed spaces $ \locTop$ \cite[Definition 1.1.5.1]{LurieSAG} by the formula
\[
\bB G \mapsto \cX [\widehat{G}] \coloneqq \Map_{\Ab^s(\SpSch_S)} (\widehat{G}, \cE)
\]
where $\widehat{G}$ is the Pontryagin dual of $G$. As $\locTop$ is cocomplete, we obtain a functor
\[
\xymatrix{
\abcat{\Orb} \ar[r] \ar[d] & \locTop \\
\cP(\abcat{\Orb}) \ar@{..>}[ru]
}
\]
by the left Kan extension.
Finally, by precomposing with the restriction $\cP(\Orb) \to \cP(\abcat{\Orb})$, one gets the functor
\[
 \Ell_{C/S} \colon \cP(\Orb) \to \locTop.
\]
From the construction, we can identify the images
\[
\Ell_{C/S}(\bB S^1) \simeq C, \ \Ell_{C/S}(\bB C_n) \simeq C[n].
\]

We can now define the $G$-equivariant elliptic cohomology functor for each compact Lie group $G$. Write $\Ell_{C/S}(\bB G)$ for the resulting spectral scheme. For a finite $G$-space $X$, the functor in Proposition \ref{prop:Orbembedding} embeds $X$ into $\cP(\Orb)_{/ \bB G}$; applying $\Ell_{C/S}$ yields a spectral scheme $\Ell_{C/S}(X)$ over $\Ell_{C/S}(\bB G)$.
\begin{thm}[\cite{GepnerMeierNEW}]
Assume that a compact Lie group $G$ satisfies either
\begin{itemize}
    \item $G$ is abelian, or
    \item $\pi_1 G$ is torsion-free.
\end{itemize} Then for any finite $G$-space $X$, the locally ringed space $\Ell_{C/S}(X)$ is a spectral scheme.
\end{thm}
Pushing forward the structure sheaf of $\Ell_{C/S}(X)$ to $\Ell_{C/S}(\bB G)$ yields a quasi-coherent sheaf on $\Ell_{C/S}(\bB G)$. We summarize this construction as follows:
\begin{thm}
    Let $G$ be a compact Lie group with $\pi_1 G$ torsion-free, or assume $G$ is abelian. Let $p \colon C \to S$ be an oriented elliptic curve. There is a $G$-equivariant elliptic cohomology functor
    \[
    \cEll_G^{C/S} \colon \cS^{G, \fin} \to \opcat{\QCoh(\Ell_{C/S}(\bB G))}
    \]
    satisfying the following properties:
    \begin{itemize}
        \item $\cEll_G$ preserves colimits,
        \item for each orbit $G / H$ for a closed subgroup $H \triangleleft G$, we have
        \[
            \cEll^{C/S}_G(G/H) \simeq (i_H)_* \cO_{\Ell_{C/S}(\bB H)}
        \]
        where $i_H \colon \Ell_{C/S}(\bB H) \to \Ell_{C/S}(\bB G)$ is the image of the map $\bB H \to \bB G$.
    \end{itemize}
\end{thm}
Since the target category $\QCoh(\Ell_{C/S}(\bB G))$ is pointed, we can extend the domain of $\cEll_G$ along $(-)_+ \colon \cS^{G,\fin} \to \cS^{G, \fin}_*$ and obtain the reduced cohomology functor
\[
\widetilde{\cEll}_G^{C/S}\colon \cS^{G, \fin}_* \to \opcat{\QCoh(\Ell_{C/S}(\bB G))}
\]
as well. Moreover, when $G$ is a torus, \cite[Theorem 8.1, Proposition 9.2]{gepner2023equivariant} shows that $\cEll^{C/S}_G$ is symmetric monoidal:
\begin{thm}
    The functors $\cEll_G^{C/S}$ and $\widetilde{\cEll}_G^{C/S}$ are symmetric monoidal when $G = \bT^n$ for some $n \in \bN$. Moreover, the domain of $\widetilde{\cEll_{\bT^n}}$ can be extended along $\Sigma^\infty \colon \cS^{\bT^n,\fin}_* \to \Sp^{\fin}_{\bT^n}$, and we obtain a functor
    \[
    \cEll_{\bT^n} \colon \Sp_{\bT^n} \to \opcat{\QCoh(\Ell_{C/S}(\bB \bT^n))} \simeq \opcat{\QCoh(C^{\times n})}.
    \]
\end{thm}
The main interest in this paper is the circle-equivariant case $G = S^1$. Rather than working with individual oriented elliptic curves and their associated elliptic cohomology theories, we pass to the universal object. Recall that Lurie constructed the moduli of oriented elliptic curves $\orM$ \cite{LurieElliptic2} and showed an equivalence
\[
\TMF \simeq \Gamma ( \orM , \cO_{\orM} ).
\]
We can express this moduli stack as the colimit of affines
\[
\orM \simeq \colim_{\substack{\Spec R \to \orM \\ R \text{: even periodic}}} \Spec R.
\]
The universal oriented elliptic curve is obtained by gluing oriented elliptic curves along the same diagram
\[
\cE \coloneqq \colim_{\substack{\Spec R \to \orM \\ C \to \Spec R \text{: oriented} }} C
\]
Because equivariant elliptic cohomology is functorial with respect to oriented elliptic curves, we can construct the universal circle-equivariant elliptic cohomology as follows.
\begin{thm}\label{thm:Tequivcoh}
Let $p \colon \cE \to \orM $ be the universal oriented elliptic curve. Then there is an associated circle-equivariant $\TMF$-cohomology functor 
\[
\cEll_{\bT} \colon \Sp^{\fin}_\bT \to \opcat{\QCoh (\cE)}
\]
satisfying the following properties:
    \begin{enumerate}
        \item $\cEll_{\bT}$ is a colimit-preserving symmetric monoidal functor,
        \item $\cEll_{\bT} (S^0) \simeq \cO_C$,
        \item $\cEll_{\bT} (\Sigma^\infty_+(\bT / \bT [n])) \simeq (e_n)_* \cO_{C[n]}$ where $e_n \colon C[n] \to C$ is the inclusion of $n$-torsion points of $C$.
\end{enumerate}
\end{thm}
\section{The Descent Spectral Sequence: Setup}\label{sec:setup}
To define the topological Jacobi forms, we set up a space with an $S^1$-action.
Let $\rho \colon \bT = U(1) \hookrightarrow \bC^{\times}$ be the fundamental representation of the circle $\bT$. Then there is a cofiber sequence
\[ \bT_+ \to S^0 \to S^{\rho}, \]
where $S^{\rho}$ is the representation sphere associated to $\rho$. Applying $\cEll_\bT$ to this sequence, we obtain a fiber sequence of sheaves
\[
\cE ll_{\bT} (\bT_+) \simeq e_* \cO_S \leftarrow \cO_C \leftarrow \cE ll_{\bT} (S^{\rho}),
\]
and we see that $\cEll_\bT(S^\rho)$ is the invertible sheaf $\cO_\cE (-e).$ As $\cEll_\bT$ is symmetric monoidal, it sends the Spanier-Whitehead dual $S^{-\rho}$ in $\Sp_{\bT}$ to $ \cO_C (e)$ and therefore $S^{-m\rho}$ to $\cO_C (me)$.

\begin{defn} Let $p \colon \cE \to \orM$ be the universal oriented elliptic curve. The topological Jacobi form $\TJF_m$  of index $m$ is the global section 
\[
\TJF_m \coloneqq \Gamma(\cE, \cO_\cE (me)).
\]
\end{defn}
In other words, $\TJF_{m}$ is the $\bT$-equivariant elliptic homology of the representation spheres. The following two results justify this naming.
\begin{lem} Over the (classical) Deligne-Mumford stack of elliptic curves $\ellM$, we have
\[
    \pi_n \cO_\cE (me) \simeq
    \begin{cases}
        0 & (n = 2k + 1) \\
        p^* \omega^k \otimes \cO_{E} (me) & (n = 2k)
    \end{cases}
\]
where $E$ is the universal elliptic curve over $\ellM.$
\end{lem}
\begin{proof}
    Using the flatness of $p \colon \cE \to \orM$, we compute
\begin{eqnarray*}
    \pi_n \cO_\cE (me)  & \simeq \pi_n p^* \cO_{\orM} \otimes_{\pi_0 \cO_{\orM}} \pi_0 \cO_\cE (me)\\
     & \simeq \begin{cases}
         0 & (n = 2k + 1) \\
         p^* \omega^k \otimes \cO_{E} (me) & (n = 2k).
     \end{cases}
\end{eqnarray*}
\end{proof}

\begin{defn}[\cite{Kramer1995}]
    For all weights $k \in \bZ$ and indices $m \in \frac{\bZ}{2}$, the set of integral Jacobi forms is defined to be the global sections $\Gamma(E, p^* \omega^{k+2m} \otimes \cO_{E} (2me)).$
\end{defn}
\begin{rem}
    In \cite{EichlerZagier1985}, weak Jacobi form of weight $k \in \bZ$ and index $m \in \frac{\bZ}{2}$ is defined to be a two-variable function
    \[
    \phi \colon \bH \times \bC \to \bC
    \]
    such that
    \begin{itemize}
        \item for $\begin{pmatrix}
   a & b \\
   c & d
\end{pmatrix}$ in $SL_2(\bZ)$, $\phi$ satisfies
        \[
        \phi \left( \frac{a\tau + b}{c \tau + d}, \frac{z}{c \tau + d} \right) = (c\tau + d)^k e^{\frac{2\pi i m c z^2}{c \tau + d}} \phi(\tau, z),
        \]
        \item for all integers $\lambda, \mu \in \bZ$, $\phi$ satisfies
        \[
        \phi(\tau, z + \lambda \tau + \mu) = e^{-2\pi im(\lambda^2 \tau + 2 \lambda z)} \phi (\tau, z).
        \]
    \end{itemize}
    A weak Jacobi form is called integral if its Fourier coefficients are all integral. 

    The evaluation at the Tate curve gives rise to a map from the global sections $\Gamma(E, p^* \omega^{k} \otimes \cO_{E} (2me))$ to the group of weak integral Jacobi forms with integral coefficients: more detailed explanation is in Appendix \ref{App:JacobiForms}.
\end{rem}

The goal of this paper is to fully compute the descent spectral sequence associated with $\TJF_m$:
\[
H^p (E ; p^* \omega^q \otimes \cO_{E} (me)) \to \pi_{2q-p} \TJF_m.
\]

\begin{rem}
In \cite{carrick2025descentspectralsequencessynthetic}, it is shown that the descent spectral sequence for $\TMF$ is isomorphic to the signature spectral sequence of $\nu_{\MU}(\TMF)$, where $\nu_{\MU}$ is the synthetic analogue functor, and they computed the signature spectral sequence in \cite{carrick_descent_2024}. Moreover, their result extends to quasi-coherent sheaves on $\orM$, and therefore the descent spectral sequence for $\TJF_n$ is isomorphic to the signature spectral sequence for $\nu_{\MU} (\TJF_n)$. We will use this identification in the later chapters \ref{sec:TJF2} to \ref{sec:TJF7} to apply the synthetic Leibniz rule in the computation of differentials.
\end{rem}

\section{\texorpdfstring{$E_2$-term of DSS}{E2-term of DSS}}\label{sec:E2term}
From now on, we implicitly $2$-localize all spectra and spectral stacks we work with. Let $\cX$ be a spectral Deligne-Mumford stack and $\cF \in \QCoh(\cX)$ be a quasi-coherent sheaf on $\cX$. Take an \'{e}tale cover $\cU \to \cX$. Then the sheaf condition says that
\[
\Gamma (\cX ; \cF) \simeq \Tot (\Gamma (\cU \times_\cX \cU \times \cdots \times_\cX \cU ; \pi^* \cF))
\]
where $\pi \colon \cU \times_\cX \cU \times \cdots \times_\cX \cU \to \cX$ is the projection map. The descent spectral sequence is the Bousfield-Kan spectral sequence associated with this totalization. In particular, its $E_2$-term is the \v{C}ech cohomology
\[
\check{H}^p (X ; \pi_q \cF)
\]
associated to the \'{e}tale cover $U \to X$, where $U$ and $X$ are the underlying Deligne-Mumford $1$-stacks of $\cU$ and $\cX$, respectively.

In this paper, we take $X = \cE$ to be the universal oriented elliptic curve over $\orM$. Let $q \colon \cE' \to \cM^{\mathrm{or}}_{1}(3)$ be the pullback of $\cE$ along the \'{e}tale cover $\cM^{\mathrm{or}}_{1}(3) \to \orM$ from the spectral enhancement of the moduli stack of elliptic curves with a $\Gamma_{1}(3)$-structure. Since $\cM^{\mathrm{or}}_{1}(3) \to \orM$ is an \'{e}tale cover, the induced map $\cE' \to \cE$ is as well. 

Recall the identification of the underlying stack 
\[
\cM_{1}(3) \simeq [\Spec A / \bG_m]
\]
where $A = \bZ_{2}[a_{1}, a_{3}, \Delta^{\pm 1}]$, and the $\bG_m$-action is given by the degrees $ \abs{a_{1}}= 2, \abs{a_{3}} = 6$ with $\Delta = a_{3}^{3}(a_{1}^{3} - 27 a_{3})$. Then the underlying elliptic curve $E'$ on $\cM_{1}(3)$ is given by the closed subscheme of the weighted projective stack $\bP(4, 6)$ cut out by the Weierstra{\ss} equation
\[
Y^{2}Z + a_{1} XYZ + a_{3} YZ^{2} = X^{3}.
\]
\begin{lem}\label{lem:cohomology}
The higher cohomology on $E'$
\[
H^{i}(E', q^{*}\omega^{k} \otimes \cO_{E'} (me)) = 0 \  (i > 0)
\]
vanishes. Moreover,
\[
\bigoplus_{k \in \bZ} H^{0}(E', q^{*} \omega^{k} \otimes \cO_{E'} (me))
\]
is a free $A$-module of rank $m$.
\end{lem}
\begin{proof}
Note that $\cM_1(3) \simeq [\Spec A / \bG_m]$ has cohomological dimension zero. By using the Leray spectral sequence \cite[\href{https://stacks.math.columbia.edu/tag/0782}{Tag 0782}]{stacks-project}, we see that the $i$-th cohomology of $\cO_{E'}(me) \otimes q^* \omega^k$ is isomorphic to the global sections of the $i$-th derived functor $R^i q_*(\cO_{E'}(me) \otimes q^* \omega^k)$. However, by the projection formula, we have
\[
R^i q_*(\cO_{E'}(me) \otimes q^* \omega^k) \simeq R^i(q_* \cO_{E'}(me) \otimes \omega^k),
\]
and therefore the first claim follows from the semicontinuity theorem. Moreover, the global sections of $q_* \cO_{E'}(me) \otimes \omega^k$ are the degree-$k$ part of the global sections of the pullback of $q_* \cO_{E'}(me)$ to $\Spec A$. Note that the sheaf $q_* \cO_{E'}(me)$ can be extended to $\Spec \bZ_{(2)}[a_1, a_3]$, and because it is a polynomial ring over a PID, the global sections of $q_* \cO_{E'}(me)$ on $\Spec A$ turn out to be free, and the second claim follows (or apply the global Horrocks extension theorem).
\end{proof}

Denote the $A$-module generators of $\bigoplus_k H^{0}(E', q^{*} \omega^{k} \otimes \cO_{E'} (2e))$ as $1$ and $x$, and the generators of $\bigoplus_k H^{0}(E', q^{*} \omega^{k} \otimes \cO_{E'} (3e))$ as $1, x$, and $y$. Degrees are $\abs{x} = 4$ and $\abs{y} = 6.$ These $x$ and $y$ satisfy the equation $y^{2} + a_{1} xy + a_{3} y = x^{3},$ and the monomials $x^i y^j$ in $x, y$ generate $\bigoplus_k H^{0}(E', q^{*} \omega^{k} \otimes \cO_{E'} (me))$ as an $A$-module for each $m$. 

Similarly, consider the pullback of $\cE$ to the intersection $\cM_{1}(3) \times_{\orM} \cM_{1}(3)$. Recall that the underlying stack of $\cM_{1}(3) \times_{\orM} \cM_{1}(3)$ is equivalent to $[\Spec \Gamma / \bG_m]$ where
\[
\Gamma \coloneqq A[s, t] / I
\]
and $I$ is the ideal in $A$ generated by 
\begin{gather*}
s^{4} - 6st + a_{1} s^{3} -3a_{1}t - 3a_{3}s, \\ 
-27t^2 + 18s^3t + 18a_1s^2t - 27a_3t - 2s^6 - 3a_1s^5 + 9a_3s^3 + a_1^3s^3 + 9a_1a_3s^2.
\end{gather*}

Recall that the pair $(A, \Gamma)$ forms a Hopf algebroid. The structure map is determined by the automorphism of the elliptic curve $y^{2} + a_{1} xy + a_{3} y = x^{3}$ by 
\begin{align*}
x &\mapsto x - \frac{1}{3}(s^{2} + a_{1}s), \\
y &\mapsto  y - sx + \frac{1}{3}(s^{3} + a_{1}s^{2}) - t,
\end{align*}
so that $a_{1}, a_{3}$ are mapped to
\begin{align*}
a_{1} &\mapsto a_{1} + 2s \\
a_{3} &\mapsto a_{3} + \frac{1}{3}(a_1s^{2} + a_{1}^2s) + 2t.
\end{align*}

Combining the argument above, we describe the $E_{1}$-term of the descent spectral sequence. Let $B_{m}$ be the collection of polynomials in $A[x, y] /( y^{2} + a_{1} xy + a_{3} y - x^{3})$ of degree less than or equal to $2m$, where the degree of $x$ is $4$ and the degree of $y$ is $6$. Note that $B_{m}$ is a free $A$-module of rank $m$. Then the $E_{1}$-term of our descent spectral sequence is the cochain complex given by the cobar complex
\[
\xymatrix{
B_{m} \ar@<-.5ex>[r] \ar@<.5ex>[r] & \Gamma \otimes_{A} B_{m} \ar@<-.5ex>[r] \ar[r] \ar@<.5ex>[r] &  \Gamma \otimes_{A}  \Gamma \otimes_{A} B_{m} \ar@<-.6ex>[r] \ar@<-.2ex>[r] \ar@<.2ex>[r] \ar@<.6ex>[r] & \cdots
}
\]
where the coaction on $x, y, a_{1}, a_{3}$ is determined by the formula above.

To compute its cohomology, we introduce a filtration in the cobar complex $(B_m, \Sigma_m)$ by letting $\abs{x} = 4$, $\abs{y} = 6$, and $\abs{a} = 0$ for any elements $a \in A$. We call this a \textit{cellular filtration}. The cohomology of $(B_m, \Sigma_m)$ can be calculated from the associated spectral sequence (algebraic Atiyah-Hirzebruch spectral sequence). Its $E_1$-term is the cohomology of the associated graded object. Because $B_m$ is a free $A$-module generated by monomials $x^iy^j$ of degree less than $m$, we see that the resulting $E_1$-term of the algebraic AHSS is a free rank-$m$ module over the $E_2$-term of the descent spectral sequence of $\TMF$. We denote the $E_2$-generator of filtration degree $n$ as $x_n$. For example, $x_0$ is represented by $1 \in A$, $x_4$ is represented by $x \in B_2$, and $x_6$ is represented by $y \in B_3$.

From the formula for the coaction on $x$ and the fact that $-(s^2 + a_1 s)/3$ represents the element $\nu$ in the $E_2$-term of the DSS for $\TMF$ (see \ref{fact_E2TMF} for notations), we see the differential $d_2(x_4) = \nu x_0$ in the algebraic AHSS. Similarly, the coaction on $y$ shows that the $d_1$-differential on $x_6$ is equal to the element represented by
\[
-sx + \frac{1}{3}(s^{3} + a_{1}s^{2}) - t,
\]
and because $s \in \Gamma$ represents $\eta \in \pi_1 \TMF$, $d_1(y)$ hits $\eta x_4$.
As we can take a generator of the associated graded
\[
B_{m+1} / B_m \simeq \begin{cases}
    A \{x^i \} & m = 2i - 1, \\
    A \{x^{i-1} y \} & m = 2i,
\end{cases}
\]
we can similarly calculate the differentials of generators $x_n$ from the Leibniz rule. Noting that $d_2(x_{16}) = 4\nu x^3 = 0$, we obtain the following table of $16$-periodic differentials.
\begin{center}
\begin{tabular}{c|c|c|c|c|c|c|c|c}
    & $x_4$ & $x_6$ & $x_8$ & $x_{10}$ & $x_{12}$ & $x_{14}$ & $x_{16}$  & $x_{18}$\\ \hline
    $d_1$ &$0$ & $\eta x_4$ & $0$ & $\eta x_8$ & $0$ &  $\eta x_{14}$ & $0$ & $\eta x_{16}$\\ \hline
    $d_2$ & $\nu x_0$ & $0$ & $2\nu x_4$ & $\nu x_6$ & $3\nu x_8$ & $2\nu x_{10}$ & $0$ & $3\nu x_{14}$ 
\end{tabular}
\end{center}
The computation of each $E_2$-term is completed in the later chapters \ref{sec:TJF2} to \ref{sec:TJF7}.

\section{\texorpdfstring{DSS for $\TJF_\infty$}{DSS for TJFinfty}}
This section is based on an idea of T. Bauer and L. Meier \cite{bauer2025topologicaljacobiforms}, and we would like to reiterate our gratitude for explaining computational methods. The originality of the proof presented here is due to them; however, since the resulting computations will be used in later sections, we include an exposition of the calculation.

We compute the descent spectral sequence for $\TJF_\infty \coloneqq \colim_{m} \TJF_{m}$. Its $E_{2}$-term is given by the cohomology of the cobar complex
\[
B \to B \otimes_{A} \Gamma \to B \otimes_{A} \Gamma \otimes_{A} \Gamma \cdots
\]
where
\begin{align*}
A &= \bZ_{2} [a_1,   a_3, \Delta^{-1}], \\
B &= A[x, y] / (y^{2} + a_{1} xy + a_{3} y - x^{3}), \\
\Gamma &= A[s, t] / I \\
\end{align*}
with coactions
\begin{align*}
x &\mapsto x - \frac{1}{3}(s^{2} + a_{1}s), \\
y &\mapsto  y - sx + \frac{1}{3}(s^{3} + a_{1}s^{2}) - t, \\
a_1 &\mapsto a_1 + 2s, \\
a_3 &\mapsto a_3 + \frac{1}{3}(a_{1}s^{2} + a_{1}^{2}s) + 2t.\\
\end{align*}

\begin{rem} 
Alternatively, one can use the fact that the global sections of $\cO_{\cE} (\infty e) = \colim_n \cO_{\cE} (ne)$ identify with the ring of functions on the punctured oriented elliptic curve, and therefore the descent spectral sequence for $\TJF_\infty$ is isomorphic to the one for the structure sheaf of the punctured oriented elliptic curve.
\end{rem}
Denote $\Sigma \coloneqq \Gamma \otimes_{A} B$. Then $(B, \Sigma)$ forms a Hopf algebroid, and the cobar complex above comes from $(B, \Sigma)$. We recall the change-of-rings theorem to compute its cohomology.
\begin{thm}[\cite{hovey_invertible_1999}, \cite{hovey2001moritatheoryhopfalgebroids}]\label{thm:cohbasechange}
Let $(A, \Gamma)$ be a Hopf algebroid and $f \colon A \to A'$ a ring map. If there exists a ring $R$ and a ring map $A' \otimes_{A} \Gamma \to R$ such that the composite
\[
\xymatrix{
A \ar[r]^-{f \otimes \eta_{R}} & A' \otimes_{A} \Gamma \ar[r] & R \\
}
\]
is faithfully flat, then the Hopf algebroid map $(A, \Gamma) \to (A', A' \otimes_{A} \Gamma \otimes_{A} A')$ induces an equivalence of comodule categories.
\end{thm}

\begin{lem}[\cite{Bauer_2008}, Remark 4.2]
    Let $(A', \Gamma')$ be the Hopf algebroid given by
    \begin{align*}
        A' &= \bZ_{(2)}[a_1, a_3, a_4, a_6, \Delta^\pm], \\
        \Gamma' &= A'[s,t]
    \end{align*}
    with coactions
    \begin{align*}
        a_1 &\mapsto a_1 + 2s, \\
        a_3 &\mapsto a_3 + a_1 r + 2t, \\
        a_4 &\mapsto a_4 - a_3s -a_1 t - a_1 rs - 2st + 3r^2,\\
        a_6 &\mapsto a_6 + a_4r - a_3 t - a_1 rt - t^2 + r^3,
    \end{align*}
    where $r = \frac{1}{3}(s^2 + a_1 s)$. Then the quotient map $A' \to A$ satisfies the condition of Theorem \ref{thm:cohbasechange}.
\end{lem}

Define the Hopf algebroid
\begin{align*}
    B' &= A'[x,y] / (y^2 + a_1 xy + a_3 y - x^3 - a_4x - a_6)\\
    &\simeq \bZ_{(2)}[a_1, a_3, a_4, x,y, \Delta^\pm] \\
    \Sigma' &= B'[s,t] \simeq B' \otimes_{A'} \Gamma'.
\end{align*}
with coactions
\begin{align*}
x &\mapsto x - \frac{1}{3}(s^{2} + a_{1}s), \\
y &\mapsto  y - sx + \frac{1}{3}(s^{3} + a_{1}s^{2}) - t. \\
\end{align*}
Then Theorem \ref{thm:cohbasechange} tells us that the cohomology $H^*(B, \Sigma)$ is isomorphic to $H^*(B', \Sigma')$. We compute $H^*(B', \Sigma')$ with further applications of Theorem \ref{thm:cohbasechange}.
\begin{lem}
Consider the quotient map $q \colon B' \to B' / (y)$. Then the composition
\[
\xymatrix{
B' \ar[r]^-{q \otimes \eta_{R}} & B' / (y) \otimes_{B'} \Sigma'
}
\]
is faithfully flat.
\end{lem}
\begin{proof}
By definition, the ring map
\[
\bZ_{2}[a_{1}, a_{3}, a_4,  x, y, \Delta^\pm] \to \bZ_{2}[a_{1}, a_{3}, a_4, x, s, t, \Delta^\pm]
\]
is determined by sending $a_{1}, a_{3}, a_4, x$ to its coactions and 
\[
y \mapsto -sx + \frac{1}{3}(s^{3} + a_{1} s^2) -t.
\]

Denote
\begin{align*}
    a_1' &\coloneqq a_1 + 2s \\
    a_3' &\coloneqq a_3 + \frac{1}{3}(a_{1}s^{2} + a_{1}^{2}s) + 2t \\
    x' &\coloneqq x - \frac{1}{3}(s^{2} + a_{1}s).
\end{align*}
Then $\bZ_{(2)}[a_{1}, a_{3}, x, s, t] $ is isomorphic to $\bZ_{(2)}[a_{1}', a_{3}', x', s, t]$ and the ring map we consider is determined by 
\begin{align*}
    a_1 &\mapsto a_1' ,\\
    a_3 &\mapsto a_3', \\
    x &\mapsto x', \\
    y & \mapsto -sx' - t.
\end{align*}
  The claim follows because $-sx' -t$ is monic in $t$.
\end{proof}
\begin{rem}
Note that the ring $B$ classifies an elliptic curve together with a non-unital point on it. With a suitable choice of $t$, we can move any point to one whose $y$-coordinate is $0$ via an automorphism of the elliptic curve. Therefore the stack represented by $(B, \Sigma)$ is equivalent to that represented by $(B/(y), (y) \backslash \Sigma / (y))$.
\end{rem}
Therefore, the cohomology we want comes from the Hopf algebroid
\begin{align*}
    B^{(2)} &\coloneqq B' / (y) = \bZ_{(2)} [a_1, a_3, a_4, x, \Delta^\pm], \\ 
    \Sigma^{(2)} &\coloneqq B^{(2)} \otimes_{B} \Sigma' \otimes_{B} B^{(2)} = B^{(2)}[s,t] / (-sx+sr - t) \simeq B^{(2)}[s] \\
\end{align*}

with coactions
\begin{align*}
a_{1} &\mapsto a_{1} + 2s, \\
a_{3} &\mapsto a_{3} -2sx + \frac{1}{3}(a_1 + 2s)(s^2 + a_1 s) ,\\
a_4 &\mapsto a_4 - a_3 s -a_1 (sr-sx) - a_1 rs - 2s(sr-sx) + 3r^2, \\
x &\mapsto x - \frac{1}{3}(s^2 + a_1 s).
\end{align*}
We can simplify the Hopf algebroid with a further application of the change-of-rings theorem.
\begin{lem}
    Consider the quotient map $q \colon B^{(2)} \to B^{(2)} / (x)$. Then the composition
\[
\xymatrix{
B^{(2)} \ar[r]^-{q \otimes \eta_{R}} & B^{(2)} / (x) \otimes_{B'} \Sigma^{(2)}
}
\]
is faithfully flat.
\end{lem}
\begin{proof}
    We can apply a similar argument above, noting that $-(s^2 + a_1 s)/3$ is a monic polynomial in $s$.
\end{proof}
Therefore the cohomology of $(B, \Sigma)$ is isomorphic to that of $(B^{(3)}, \Sigma^{(3)})$, where
\begin{align*}
    B^{(3)} &\coloneqq B^{(2)} / (x) = \bZ_{(2)}[a_1, a_3, a_4, \Delta^{\pm}], \\
    \Sigma^{(3)} &\coloneqq B^{(3)} \otimes_{B^{(2)}} \Sigma^{(2)} \otimes_{B^{(2)}} B^{(3)} = B^{(3)}[s] / (s^2 + a_1s), \\
\end{align*}
with coactions
\begin{align*}
    a_1 &\mapsto a_1 + 2s, \\
    a_3 &\mapsto a_3, \\
    a_4 &\mapsto a_4 - a_3s.
\end{align*}

We use the algebraic Novikov spectral sequence associated with an invariant ideal to compute its cohomology. Note that $I_0 \coloneqq (2), I_1 \coloneqq (2, a_1), I_2 \coloneqq (2, a_1, a_3)$ are invariant ideals in $(B^{(3)}, \Sigma^{(3)})$. Denote the quotient Hopf algebra $(B^{(3)}/I_i, I_i \backslash \Sigma^{(3)} / I_i)$ as $(B_i, \Sigma_i)$ for $i = 0,1,2.$ 

Note that $(B_2, \Sigma_2) = (\bF_2 [a_4], B_2[s]/(s^2))$ with the trivial coaction. Therefore we see that $H^{*}(B_{2}, \Sigma_{2}) \simeq \bF_{2}[\eta, a_4]$ where $\eta \in H^1$ is the element represented by $s$. The coaction says that $a_3, a_4^2$ are permanent cycles, and $a_4$ hits $a_3 \eta$. From the algebraic Novikov spectral sequence, we obtain 
\[
H^{*}(B_{1}, \Sigma_{1}) \simeq \bF_{2}[\eta, a_{3}, a_4^2] / (a_3 \eta).
\] 
Furthermore, because the coaction of $a_1$ is trivial mod $2$ and $a_4^2 + a_1a_3a_4$ is a cocycle, we have
\[
H^*(B_0, \Sigma_0) \simeq \bF_2 [\eta, a_1, a_3, a_4^2+a_1 a_3 a_4] / (a_3 \eta).
\]
Now we run the $2$-Bockstein spectral sequence. Differentials are determined by $a_{1} \mapsto 2s$. Note that $a_{1}^{2}$ and $a_1 a_3 + 2a_4$ are cocycles. Because $\eta$ times $a_4^2 + a_1 a_3 a_4$ is nonzero and $\eta$ is $2$-torsion, $a_1$ times $a_4^2 + a_1 a_3 a_4$ must hit $2\eta \cdot (a_4^2 + a_1 a_3 a_4)$. Therefore we conclude
\[
H^*(B, \Sigma) = \bZ_{(2)} [\eta, a_1^2, a_3, a_1a_3 + 2a_4, a_4^2 + a_1a_3a_4] / (2\eta, a_3\eta).
\]

Recall that the ring of meromorphic integral modular forms $\MF_*$ is isomorphic to $\bZ[c_4, c_6, \Delta^{\pm}] / (c_4^3 - c_6^2 - 1728\Delta)$, where $c_4, c_6,$ and $\Delta$ are given by
\begin{align*}
c_4 &\coloneqq a_1^4 -24a_1a_3 -48a_4, \\
c_6 &\coloneqq a_1^6 -36a_1^3a_3-72a_1^2a_4+216a_3^2, \\
\Delta &\coloneqq a_1^5a_3a_4 + a_1^3a_3^3 + a_1^4a_4^2 - 30a_1^2a_3^2a_4 - 27a_3^4 - 96a_1a_3a_4^2 -64a_4^3.
\end{align*}
Therefore, as a module over $\MF_*$, the $E_2$-term of DSS for $\TJF_\infty$ can be written as
\[
 \MF_* [\eta, b,c,d,e]/(2\eta, c\eta, d\eta, 4e - d^2 + bc^2, c_4^3-c_6^2 - 1728\Delta, b^2 - 24d- c_4, b^3 - 36bd+216c^2-c_6 ).
\]
 

The remaining task is to work out the higher differentials of the descent spectral sequence for $\TJF_\infty$. Because $\eta \in H^1$ lies in the image of the unit map $\TMF \to \TJF_\infty$, it must satisfy the relation $\eta^4 = 0 \in \pi_* \TJF_\infty$. The only class that can hit $\eta^4$ is $b \eta$, so we obtain
\[
d_3 (b) = \eta^3.
\]
For degree reasons, there are no other possible differentials, and we finish the computation of $\pi_* \TJF_\infty$. A part of the spectral sequence is drawn in Figure \ref{fig:TJFinfty}.

\begin{figure}[tp]
\centering
    \fullpagegraphic{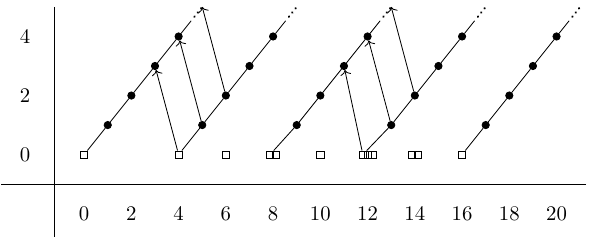}
    \caption{The $E_3$-page of DSS for $\TJF_{\infty}$ and differentials}
    \label{fig:TJFinfty}
\end{figure}

\begin{rem}
Although the homotopy groups of $\TJF_\infty$ only have $\KO$-like elements and free parts, we do \emph{not} have a spectrum decomposition of $\TJF_\infty$ into the sum of a shifted copy of $\KO$ and $\KU$. In fact, $\TJF_\infty$ is not even $K(2)$-acyclic. One way to see this is the following: consider the inclusion $\cE[3] \to \cE$. Removing the basepoint of those schemes and taking the global section, we obtain a ring morphism $\TJF_\infty \to \TMF_1(3)$ by to Meier's decomposition of $\TMF^{C_3}$ \cite{meier2022topologicalmodularformslevel}, and the latter is not $K(2)$-acyclic.
\end{rem}

\section{\texorpdfstring{Transfers and Cell Structures of $\TJF$}{Transfers and Cell Structures of TJF}}\label{sec:cellstr}
We make the following observation to compute the higher differentials of $\TJF_{m}$. The main goal of this section is to prove Theorem \ref{celldecomp}.
\begin{defn}\label{def:Pm}
    Let $m \geq 1$ be an integer, $\tr \colon \Sigma\Sigma^{\infty}_+ \bC P^{\infty} \to S^0$ be the circle-equivariant transfer map, and $q \colon \bC P^\infty_+ \to S^0$ be the collapsing map. Define a finite spectrum $P_m$ to be the cofiber of the map
    \[
    \tr \oplus \Sigma q \colon \Sigma \bC P^{m-1}_+ \to S^0 \oplus S^1.
    \]
\end{defn}
\begin{rem}
    Recall from \cite[Lemma 2.7]{MillerBernoulli} that the cofiber of the circle-equivariant transfer map is equivalent to the double suspension of the stunted projective space $\Sigma^{2} \bC P^{\infty}_{-1}$. Therefore we obtain the following commutative diagram of cofiber sequences.
    \[
    \xymatrix{
    \Sigma \bC P^{m-1}_+ \ar[r] \ar@{=}[d] & S^0 \oplus S^1 \ar[r] \ar[d] & P_m \ar[d]\\
    \Sigma \bC P^{m-1}_+ \ar[r] \ar[d] & S^0 \ar[r] \ar[d] & \Sigma^{2} \bC P^{m-1}_{-1} \ar[d] \\
    0 \ar[r] & S^2 \ar@{=}[r] & S^2
    }
    \]
    The right vertical sequence exhibits $P_m$ as “$\Sigma^2 \bC P^{m-1}_{-1}$ with the $2$-cell removed.” In particular, $P_{m}$ has one cell in every even dimension except $2$, and $P_{2}$ is equivalent to the cofiber of the Hopf fibration $C\nu = \bH P^{2}$. 
\end{rem}
\begin{thm}\label{celldecomp}
 For each integer $m \geq 1$, we have an equivalence of $\TMF$-modules
\[
\TJF_{m} \simeq \TMF \otimes P_{m}.
\]
\end{thm}
This theorem is heavily used in the computation of the descent spectral sequence because the spectral sequence for $\tmf \otimes P_{m}$ has far fewer possible differentials, and we obtain the one for $\TJF_m$ by inverting the polynomial generator $\Delta^{8} \in \pi_{192} \tmf$. 

To prove the theorem, we recall one observation in spectral algebraic geometry.

\begin{lem}
    The dualizing sheaf $\omega_{C / S}$ for an oriented elliptic curve $p \colon C \to S$ is isomorphic to $S^{-1} \otimes \cO_{C}.$ For a quasi-coherent sheaf $\cF \in \QCoh(C),$ the global section of the dual sheaf $\Gamma(C, \cF^\vee)$ is isomorphic to the degree-shifted $\TMF$-dual $\Sigma D ( \Gamma(C, \cF)).$
\end{lem}
\begin{proof}
The first assertion follows from Example 6.4.2.9 of \cite{LurieSAG} and Remark 5.2.5 of \cite{LurieElliptic1}. The second statement follows from the observation
    \begin{align*}
        \Hom_S (p_* \cF, \cO_S) &\simeq \Hom_C (\cF, \omega_{C/S}) \\
        &\simeq \Hom_C (\cF, S^{-1} \otimes \cO_C) \\
        &\simeq S^{-1} \otimes \Hom_C (\cF, \cO_C) \\
        &\simeq S^{-1} \otimes \Gamma(C, \cF^\vee).
    \end{align*}
\end{proof}
Consider the sequence of pointed spaces
\[
S(m \rho)_{+} \to S^{0} \to S^{m \rho}.
\]
where $S(m \rho)$ is the unit sphere in the representation $m \rho$. Applying $\widetilde{\cE ll}_{\bT}$, we obtain a fiber sequence
\[
\widetilde{\cE ll}_{\bT}(S(m \rho)_+) \leftarrow \cO_C \leftarrow \cO_C (-me)
\]
in $\QCoh(C)$. Since the group $\bT$ acts freely on $S(m\rho)$, we have
\[
\widetilde{\cE ll}_{\bT}(S(m \rho)_+) \simeq e_* (\cO_S^{\bC P^{m-1}})
\]
by (3) of Theorem \ref{thm:Tequivcoh}.
Taking duals of this sequence and global sections, one obtains the cofiber sequence 
\[
\TMF \otimes \Sigma\bC P^{m-1} \to \TMF_\bT \to \TJF_m.
\]
of $\TMF$-modules.
Note that $\Sigma \TMF \hookrightarrow \TMF \otimes \Sigma \bC P^{m-1} \to \TMF_\bT$ is the transfer map defined in \cite{gepner2023equivariant}. They showed that the cofiber of the transfer map $\Sigma \TMF \to \TMF_\bT$ is equivalent to $\TMF$. In particular, $\TJF_m$ is the cofiber of the map
\[  
\TMF \otimes \Sigma \bC P^{m-1}_1 \to \TMF \to \TJF_m.
\]

We induct on $m$ using the cofiber sequence above to prove the theorem. First, we know that $\TJF_{1} = \Gamma(C, \cO_{C}(e)) \simeq \TMF$. (Maybe cite GM here). Notice the following commutative diagram:
\[
\begin{xymatrix}{
\TMF \otimes \Sigma S^{2m-1} \ar[r] \ar[d] & 0 \ar[r]\ar[d] & \TMF \otimes \Sigma^2 S^{2m-1} \ar[d] \\
\TMF \otimes \Sigma \bC P^{m-1}_{1} \ar[r] \ar[d] & \TMF \ar[d]_{\text{id}} \ar[r] & \TJF_{m} \ar[d] \\
\TMF \otimes \Sigma \bC P^{m}_{1} \ar[r] & \TMF  \ar[r] & \TJF_{m+1} \\
}
\end{xymatrix}
\]
where both rows and columns are cofiber sequences. The right vertical sequence shows that $\TJF_{m+1}$ is obtained by attaching the $S^{2m+2}$-cell to $\TJF_m \simeq \TMF \otimes P_m$. We analyze this attaching from the cobar complex described in section \ref{sec:E2term}. 

We start with the base case $m = 1.$ The $E_2$-term of the descent spectral sequence for $\TJF_2$ is the cohomology of the Hopf algebroid $(B_2, B_2 \otimes_A \Gamma)$, and $B_2$ is the free $A$-module of rank $2$ generated by $1$ and $x.$ The coaction on $x$ tells us that the $d_1$-differential on $x$ is given by
\[
d_1(x) = -\frac{1}{3}(s^{2} + a_{1}s),
\]
and $-\frac{1}{3}(s^{2} + a_{1}s)$ represents $\nu$ in the $E_2$-page of DSS for $\TMF$. Therefore $\nu \in \TMF$ becomes zero in $\TJF_2$, and this fact forces the map $\TMF \otimes S^3 \to \TJF_1 = \TMF$ to be equal to $\nu \in [\Sigma^3 \TMF, \TMF]_{\TMF} \simeq \pi_3 \TMF$. In particular, $\TJF_2$ is equivalent to the cofiber $\TMF \otimes C\nu.$

We proceed with the induction as follows. Assume the equivalence $\TJF_n \simeq \TMF \otimes P_m$, and consider the following diagram:
\[
\begin{xymatrix}{  
    \TMF \otimes \Sigma^2 S^{2m-1} \ar[r] \ar@{=}[d] & \TMF \otimes P_m \ar[r] \ar[d] & \TMF \otimes P_{m+1} \\
    \TMF \otimes \Sigma^2 S^{2m-1} \ar[r] \ar@{=}[d] & \TJF_m \ar[r] \ar[d] & \TJF_{m+1} \ar[d] \\
    \TMF \otimes \Sigma^2 S^{2m-1} \ar[r]  & \TMF \otimes \Sigma^2 \bC P^{m-1}_1 \ar[r] & \TMF \otimes \Sigma^2 \bC P^m_{1}. \\
    }
\end{xymatrix}
\]
The bottom squares are commutative, and the composition of the middle vertical map is given by the collapse of the bottom $0$-cell. It suffices to show that the top left square commutes to finish the induction. Because the two maps coincide after the composition $\TJF_m \to \TMF \otimes \Sigma^2 \bC P^{m-1}_1$, we can check the commutativity by the diagram chase from the below.

\[
\begin{xymatrix}{  
    0 \ar[r] \ar[d] & \TMF \ar@{=}[r] \ar[d] & \TMF \ar[d] \\ 
    \TMF \otimes \Sigma^2 S^{2m-1} \ar[r] \ar@{=}[d] & \TMF \otimes P_m \ar[r] \ar[d] & \TMF \otimes P_{m+1} \ar[d] \\
    \TMF \otimes \Sigma^2 S^{2m-1} \ar[r]  & \TMF \otimes \Sigma^2 \bC P^{m-1}_1 \ar[r] & \TMF \otimes \Sigma^2 \bC P^m_{1} \\
    }
\end{xymatrix}
\]

As a consequence of the cell decomposition, we can deduce a horizontal vanishing line of DSS for $\TJF_n$. 

\begin{lem}{(see also \cite{Beaudry_2022} and \cite{Chua_2022})}\label{lem:vanishing}
The descent spectral sequence for $\TJF_n$ has the horizontal vanishing line at $s = 24$ in $E_{24}$-page.
\end{lem}

\begin{proof}
Denote by $I$ the kernel of the map $\TMF \to \TMF_1(3)$. The descent spectral sequence for $\TMF$ has a horizontal vanishing line at $s = 24$ in the $E_{24}$-page \cite{Bauer_2008}, and therefore the map $I^{\otimes 24} \to \TMF$ is phantom. As $\TMF$ is the limit of even periodic $\bE_\infty$-rings, any phantom map to $\TMF$ is zero. Therefore the map $I^{\otimes 24} \to \TMF$ is zero too.

The argument in this section identifies the DSS for $\TJF_n$ with the Adams-Novikov spectral sequence of $\TMF \otimes P_n$ by the map $\TMF \to \TMF_1(3).$ As we know the map $I^{\otimes 24} \otimes P_n \to \TMF \otimes P_n$ is zero, we have the stated vanishing line from \cite[Proposition 2.29]{mathew2017examplesdescentnilpotence}.
\end{proof}

\section{\texorpdfstring{$\TJF_2$}{TJF2}}\label{sec:TJF2}

\subsection{\texorpdfstring{$E_2$-term}{E2-term}}
We determine the higher differentials and the $E_{\infty}$-page of the descent spectral sequence for $\TJF_{m}$. We start by recalling the $E_{2}$-term of the descent spectral sequence for $\TMF$, computed in \cite{Bauer_2008} and shown in Figure~\ref{fig:TMF2}.

\begin{thm}[\cite{behrens2019topologicalmodularautomorphicforms}]\label{fact_E2TMF}
    As a ring, the $E_2$-page of the descent spectral sequence for $\TMF$ is isomorphic to
    \[
    E_2^{*, *} \simeq \bZ_{(2)}[c_4, c_6, \Delta^{\pm 1}, \eta, \nu, \delta, \epsilon, \kappa, \bkappa] / I
    \]
    where $I$ is the ideal generated by
    \[
        2\eta, \eta \nu, 4\nu, 2\nu^2, \nu^3 - \eta \epsilon,
     \]
     \[
        2\epsilon, \nu \epsilon, \epsilon^2, 2\delta, \nu \delta, \epsilon \delta, \delta^2 = c_4 \eta^2,
    \]
    \[
    2\kappa, \eta^2 \kappa, \nu^2 \kappa - 4 \bkappa, \epsilon \kappa, \kappa^2, \kappa \delta, 
    \]
    \[
    \nu c_4, \nu c_6, \epsilon c_4, \epsilon c_6, \delta c_4 - \eta c_6, \delta c_6 - \eta c_4^2,
    \]
    \[
    \kappa c_4, \kappa c_6, \bkappa c_4 - \eta^4 \Delta, \bkappa c_6 - \eta^3 \delta \Delta, c_6^2 - c_4^3 + 1728\Delta 
    \]
    with bidegrees
    \[
    \begin{array}{lll}
    \abs{c_4} = (0,8), & \abs{c_6} = (0,12), & \abs{\Delta} = (0,24), \\ 
    \abs{\eta} = (1, 2), &  \abs{\nu} = (1,4), & \abs{\delta} = (1,6), \\
    \abs{\epsilon} = (2,10), & \abs{\kappa} = (2,16), & \abs{\bkappa} = (4,24).
    \end{array}
    \]
\end{thm}

\begin{figure}[ht]
\centering
    \fullpagegraphic{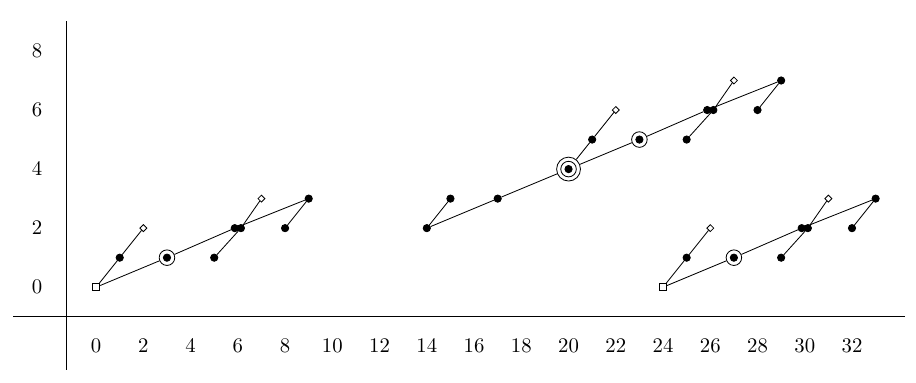}
    \caption{The $E_2$-page of DSS for the 2-local $\TMF$}
    \label{fig:TMF2}
\end{figure}

By the cell structure in Section~\ref{sec:cellstr}, the $E_1$-term of the algebraic AHSS has two copies of the $E_{2}$-term of $\TMF$, generated by $x_0$ in bidegree $(0,0)$ and $x_4$ in bidegree $(4,0)$. The $d_{1}$-differential sends $x_4 \mapsto \nu \cdot x_0$.

\begin{figure}[ht]
\centering
    \fullpagegraphic{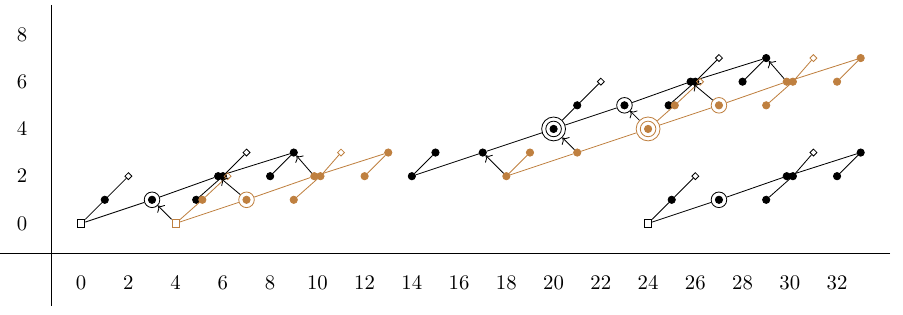}
    \caption{$E_1$-term of algebraic AHSS for $\TJF_2$ and differentials}
\end{figure}

This gives the $E_{2}$-page of the DSS for $\TJF_{2}$ (Figure~\ref{fig:TJF2E2}); it is $\Delta$- and $\bkappa$-linear.

\begin{figure}[ht]
\centering
    \fullpagegraphic{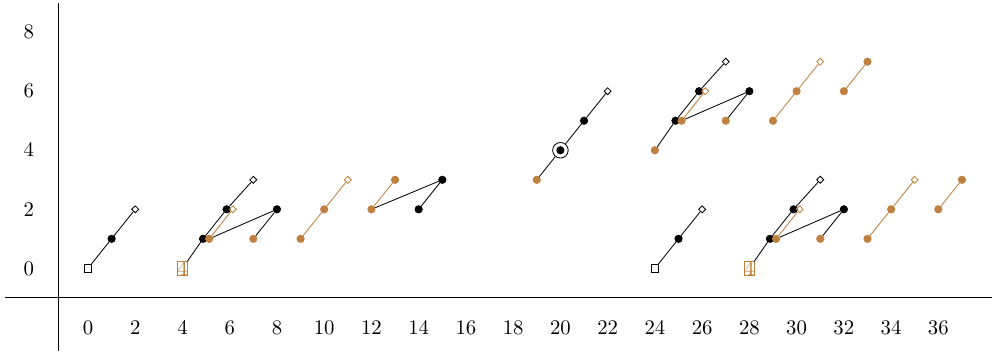}
    \caption{The $E_{2}$-page of DSS for $\TJF_{2}$}
    \label{fig:TJF2E2}
\end{figure}

We note the following multiplicative extensions.

\begin{lem}
There are the following multiplicative extensions in the $E_2$-term.
\begin{enumerate}
\item $\eta \cdot [4x_4] = [\delta x_0]$,
\item $\nu \cdot [\eta x_4] = [\epsilon x_0],$
\item $\eta \cdot [2\nu x_4] = [\epsilon x_0],$
\item $\nu \cdot [\epsilon x_4] = [\eta \kappa x_0],$
\item $\eta \kappa \cdot [\eta x_4] = [2 \bkappa x_0].$
\end{enumerate}
\end{lem}

\begin{proof}
We use Massey brackets in the $E_2$-term and shuffling lemma:
\begin{enumerate}
\item $\eta \cdot [4x_4] = \eta  \langle 4, \nu, x_0 \rangle = \langle \eta, 4, \nu \rangle x_0 = [\delta x_0],$
\item $\nu \cdot [\eta x_4] = \nu \langle \eta, \nu, x_0 \rangle = \langle \nu, \eta, \nu \rangle x_0 = [\epsilon x_0],$
\item $\eta \cdot [2\nu x_4] = \eta \langle 2 \nu, \nu, x_0 \rangle = \langle \eta, 2\nu, \nu \rangle x_0 = [\epsilon x_0],$
\item $\nu \cdot [\epsilon x_4] = \nu \langle \epsilon, \nu, x_0 \rangle = \langle \nu, \epsilon, \nu \rangle x_0 = [\eta \kappa x_0],$
\item $\eta \kappa \cdot [\eta x_4] = \eta \kappa \langle \eta, \nu, x_0 \rangle = \langle \eta \kappa, \eta, \nu \rangle x_0 = [2 \bkappa x_0].$
\end{enumerate}
\end{proof}

\subsection{Higher Differentials}
The map $\TMF \to \TJF_2$ induces the $d_{3}$-differential $d_3 ( [2\nu  x_0]) = \eta^4 x_0$, and $\eta$-linearity gives $d_3([4x_4]) = \eta^3 x_0$. No other $d_{3}$-differentials occur, so the $E_{4}$-page is as in Figure \ref{fig:TJF2higher1}.

\begin{rem}
    Since any element above the $3$-line is divisible by $\bkappa$, it suffices to look at elements with cohomological degree at most $3$ when computing higher $d_r$. We use the same observation for $\TJF_n$.
\end{rem}

For bidegree reasons only odd $r$ can appear. Thus it is enough to determine $d_5, d_7, \cdots, d_{23}$.

\begin{lem}
    The $d_5$-differentials are $\Delta^2$-linear and are determined by
    \begin{enumerate}
        \item $d_5([\Delta \eta x]) = [\epsilon \bkappa x_0],$
        \item $d_5([\Delta \epsilon x]) = [\eta \kappa \bkappa x_0],$
    \end{enumerate}
    together with $\Delta^2$-linearity.
\end{lem}
\begin{proof}
    Recall that $d_5(\Delta) = \nu \bkappa$ in the DSS for $\TMF.$ For $w \in E_4^{*,*}$ the Leibniz rule gives
    \[
        d_5(\Delta w) = \nu \bkappa w + \Delta d_5(w).
    \]
    Because $\nu$-multiples are $2$-torsion in the $E_2$-page, we also have
    \[
        d_5(\Delta^2 w) = 2 \Delta \nu \bkappa w + \Delta^2 d_5(w) = \Delta^2 d_5 (w),
    \]
    which yields the displayed formulas and $\Delta^2$-linearity. Sparsity rules out other $d_5$-differentials.
\end{proof}

\begin{lem}
    The $d_7$-differential is $\Delta^4$-linear and determined by
    \begin{enumerate}
        \item $d_7([\Delta^{2i+1} \eta \kappa x_4]) = [\Delta^{2i} \eta^2 \bkappa^2 x_0],$
        \item $d_7([\Delta^{2i} \eta \kappa x_4]) = 0.$
    \end{enumerate}
\end{lem}

\begin{proof}
    The equality $d_7(\Delta^4) = \Delta^3 \eta^3 \bkappa = 4 \Delta^3 \nu \bkappa$ in the DSS for $\TMF$ implies $\Delta^4$-linearity here, by the same argument as for $d_5$.

    \textbf{Proof 1.}
    The class $\eta^2 \bkappa^2 \in \pi_{42} \TMF$ is $\nu$-divisible, so $[\eta^2 \bkappa^2 x_0]$ must be the target of a differential. By $\Delta^4$-periodicity the only possible sources are $[\Delta^{2i+1} \eta \kappa x_4]$, giving the stated nontrivial $d_7$ on odd powers. For even powers,
    \[
        d_7([\Delta^2 \eta \kappa x_4]) = \kappa \, d_7([\Delta^2 \eta x_4]) = 0,
    \]
    and $\Delta^4$-linearity forces $d_7([\Delta^{2i} \eta \kappa x_4]) = 0$ for all $i$. If $d_7(\Delta^7 \eta \kappa x_4)=0$, the class $\eta^2 \bkappa^2 \Delta^6 x_0$ would have to support a later differential, forcing $d_{17}(\eta^2 \bkappa^2 \Delta^6 x_0) = \bkappa^5 \Delta^3 \eta \epsilon x_0$. Below we show $d_{9}(\bkappa^5 \Delta^3 \eta \epsilon x_4) \neq 0$, ruling this out.

    \textbf{Proof 2 (synthetic Leibniz rule).}
    Using the $d_5$-differentials and the synthetic Leibniz rule,
    \begin{align*}
       \delta_4^8([\Delta^{2i+1} \eta \kappa x_4]) &= \kappa \, \delta_4^8([\Delta^{2i+1} \eta x_4])  \\
       &= \kappa [\Delta^{2i} \epsilon \bkappa x_0] \\
       &= [\Delta^{2i} \eta^2 \bkappa^2 x_0],
    \end{align*}
    giving the same nontrivial $d_7$ on odd powers.
\end{proof}

\begin{rem}
    Recall that $w^2 = [\Delta^2 \eta \bkappa^2]$ is a permanent cycle in the DSS for $\TMF.$ Its image in the $E_\infty$-page of the DSS for $\TJF_2$ is zero, but its image in $\pi_* \TJF_2$ is nonzero and lands in higher Adams filtration.
\end{rem}

To determine differentials further, we use the following lemma (geometric boundary theorem) shown in \cite[Lemma 6.13]{Beaudry_2022}.
\begin{lem}\label{lem:naturalitytotarget}
    Let $X \to Y \xrightarrow{p} Z$ be a (co)fibration of spectra and $a \in E_r(Y)$ be an element in DSS. Suppose that $p_*(a)$ persists to $E_{r'}$-term for some $r' \geq r$ and there is a nontrivial differential $d_{r'}(p_* (a)) \neq 0.$ Then there is a nontrivial differential $d_{r''} (a) \neq 0$ for some $r'' \leq r'.$
\end{lem}

\begin{lem} The $d_9$-differentials are determined by the following differentials with $i = 0, 1$:
    \begin{enumerate}
    \item $d_9([\Delta^{2+4i} x_0]) = [\Delta^{4i} \bkappa^2 \cdot 2\nu x_4],$
    \item $d_9([\Delta^{3+4i} x_0]) = [\Delta^{4i+1} \bkappa^2 \cdot 2\nu x_4],$
    \item $d_9([\eta \Delta^2 x_4]) = [\bkappa^2 \epsilon x_4].$
\end{enumerate}
\end{lem}

\begin{proof}
    We prove the claim for $i=0.$ The same argument applies to the $i=1$ case.
    The first differential follows from the fact that
    \[
    \pi_{47} \TMF \otimes C\nu =0.
    \]
    The second one is the image of $d_9$-differential of $\TMF$ and $\eta$-linearity.
    $d_9([\eta \Delta^2 x_4])$ is the consequence of lemma \ref{lem:naturalitytotarget}.

    The first differential can also be deduced from the synthetic Leibniz rule as follows: in $\nu_\MU (\TJF_2) / \tau^{12}$, we compute
    \begin{align*}
    \tau^8 \bkappa^2 [2\nu x_4] &= \tau^8 \bkappa^2 \langle 2\nu, \nu, x_0 \rangle \\
    &= \tau^4 \bkappa \langle \tau^4 \bkappa, 2\nu, \nu \rangle x_0 \\
    &= \tau^4 \bkappa [2\nu\Delta] x_0 \\
    &= 0.
    \end{align*}
    where the last equality is proven in \cite[Corollary 6.36]{carrick_descent_2024}. Therefore $[\bkappa^2 2\nu x_4]$ must be a target of a $d_9$-differential, and $[\Delta^2 x_0]$ is the only possible domain.
\end{proof}

\begin{lem}\label{E2extTJF2}
    The $d_{11}$-differentials are determined by the following with $i = 0, 1$:
    \begin{enumerate}
        \item $d_{11}([2\Delta^{2+4i} x_0]) = [\Delta^{4i}\bkappa^2 \eta^3 x_4]$
        \item $d_{11}([\kappa\Delta^{2+4i} x_0]) = [\Delta^{4i} \eta \bkappa^3 x_0]$
        \item $d_{11}([2\Delta^{3+4i} x_0]) = [\Delta^{1+4i} \bkappa^2 \eta^3 x_4]$
        \item $d_{11}([\kappa \Delta^{3+4i} x_0]) = [\Delta^{1+4i} \eta \bkappa^3 x_0]$
        \item $d_{11}([\Delta^{4} x_0]) = [\Delta^2 \bkappa^2 \eta^3 x_4]$
        \item $d_{11}([\Delta^{5} x_0]) = [\Delta^3 \bkappa^2 \eta^3 x_4]$
    \end{enumerate}
\end{lem}

\begin{proof}
The second and fourth differentials are the image of the differential $\TMF \to \TJF_2.$ 
Every other differential follows from the following observation that
    \[
    \pi_{24i-1} \TMF \otimes C\nu = 0, \ i = 2,3,4,5,6,7.
    \]

Alternatively, we can use the synthetic Leibniz rule to deduce these differentials.
\end{proof}

\begin{lem}
    There are no $d_{13}, d_{15}, d_{17}$-differentials.
\end{lem}
\begin{proof}
    This follows from the degree reason.
\end{proof}

\begin{lem}
    The $d_{19}$-differential is determined by
    \begin{enumerate}
        \item $d_{19}([\eta \Delta^4 x_4]) = [\bkappa^5 x_0],$
        \item $d_{19}([\Delta^7 \eta^3 x_4]) = [\Delta^3 \bkappa^5 \eta^2 x_0].$
    \end{enumerate}
\end{lem}
\begin{proof}
    One way to see this is that $\bkappa^5$ is divisible by $\nu$ by an exotic extension in $\TMF.$ The second one comes from the observation that
    \[
    \pi_{175} \TMF \otimes C\nu = 0.
    \]
    Moreover, both differentials can be deduced from the lemma \ref{lem:vanishing}. Both targets must be hit by some differential because of the vanishing line, and there is only one possibility to accomplish them.
\end{proof}

\begin{lem}
    The $d_{23}$-differential is determined by the following:
    \begin{enumerate}
        \item $d_{23}([\Delta^5 \eta^2 x_4]) = [\bkappa^6 \eta x_4], $
        \item $d_{23}([\Delta^6 \eta x_0]) = [\bkappa^6 \Delta x_0],$
        \item $d_{23}([\Delta^7 \eta^2 x_0]) = [\bkappa^6 \eta \Delta x_0].$
    \end{enumerate}
\end{lem}

\begin{proof} 
    The first differential is the consequence of Lemma \ref{lem:naturalitytotarget}. The second and third follow from Lemma \ref{lem:vanishing}.
\end{proof}

\section{\texorpdfstring{$\TJF_3$}{TJF3}}\label{sec:TJF3}
\subsection{\texorpdfstring{$E_2$-term}{E2-term}}
The algebraic AHSS has $E_1$-term generated by $x_0, x_4, x_6$, with $d_1(x_6) = \eta x_4.$ The resulting $E_2$-page is shown in Figure~\ref{fig:TJF3E2}. We record the multiplicative extensions below.

\begin{prop}
    There are the following multiplicative extensions in the $E_2$-term:
    \begin{enumerate}
        \item $\eta \cdot [2 x_6] = [2\nu x_4]$,
        \item $2 \cdot [2\nu x_6] = [\delta x_4]$,
        \item $\nu \cdot [2\nu x_6] = 2 \cdot [\nu^2 x_6] = [\epsilon x_4]$,
        \item $\eta \cdot [\nu\kappa x_6] = \bkappa \cdot [4 x_4] = [\nu^2\kappa x_4]$.
    \end{enumerate}
\end{prop}
\begin{proof}
We compare with the Adams spectral sequence converging to $\tmf \otimes P_2$.
\begin{itemize}
    \item For (1): the $E_2$-term of the ASS (Figure~\ref{fig:TJF3ASS}, computed with \cite{https://doi.org/10.11582/2021.00077}) has no possible differentials in homotopy degrees $6$ and $7$, and the $\eta$-multiple of the generator in $\pi_6 \tmf \otimes P_2$ is nonzero. These facts force the first two extensions in $\pi_* \TJF_3$. (2) can be shown similarly.
    \item For (3): we compute
    \[
        \nu \cdot (2 [\nu^2 x_6]) = \nu \cdot [\epsilon x_4] = [\eta \kappa x_0].
    \]
    The differential $d_2([\nu x_6]) = [\epsilon x_0]$ in the algebraic AHSS yields the Massey product shuffling
    \begin{align*}
        2\nu \cdot [\nu^2 x_6] &= 2\nu \langle \nu, \epsilon, x_0 \rangle \\
        &= \langle 2\nu, \nu, \epsilon \rangle x_0 \\
        &= [\eta \kappa x_0],
    \end{align*}
    giving the desired identity.
    \item For (4): from the $\TMF$-cell structure of $\TJF_3$ \ref{celldecomp}, we know the cofiber of the map is $\TMF \otimes C\eta$. The $E_2$-term of DSS for $\tmf \otimes C\eta$ is computed in \cite{bauer_cpinfty} and shown in Figure~\ref{fig:E2Ceta}. We can check this extension in $\tmf \otimes C\eta$. Alternatively, we can show this extension using ASS as well: the ASS for $\tmf \otimes P_2$ has a $d_3$ hitting the $\bZ/2$ class in bidegree $(22,8)$ (for example, we can confirm this differential as the image of $d_3$ of Adams spectral sequence converging to $\pi_* P_3$, shown in \cite{Linsseq}), so $\pi_{22} \TJF_3$ is torsion free. Hence $[\nu^2 \kappa x_4]$ must be the target of a differential in DSS, and the only possibility is $d_3([\nu \kappa x_6]) = [\nu^2 \kappa x_4]$. $\eta$-linearity then forces $\eta \cdot [\nu\kappa x_6] = [\nu^2\kappa x_4]$, which matches the stated extension.
\end{itemize}
\end{proof}
\begin{figure}
    \centering
    \fullpagegraphic{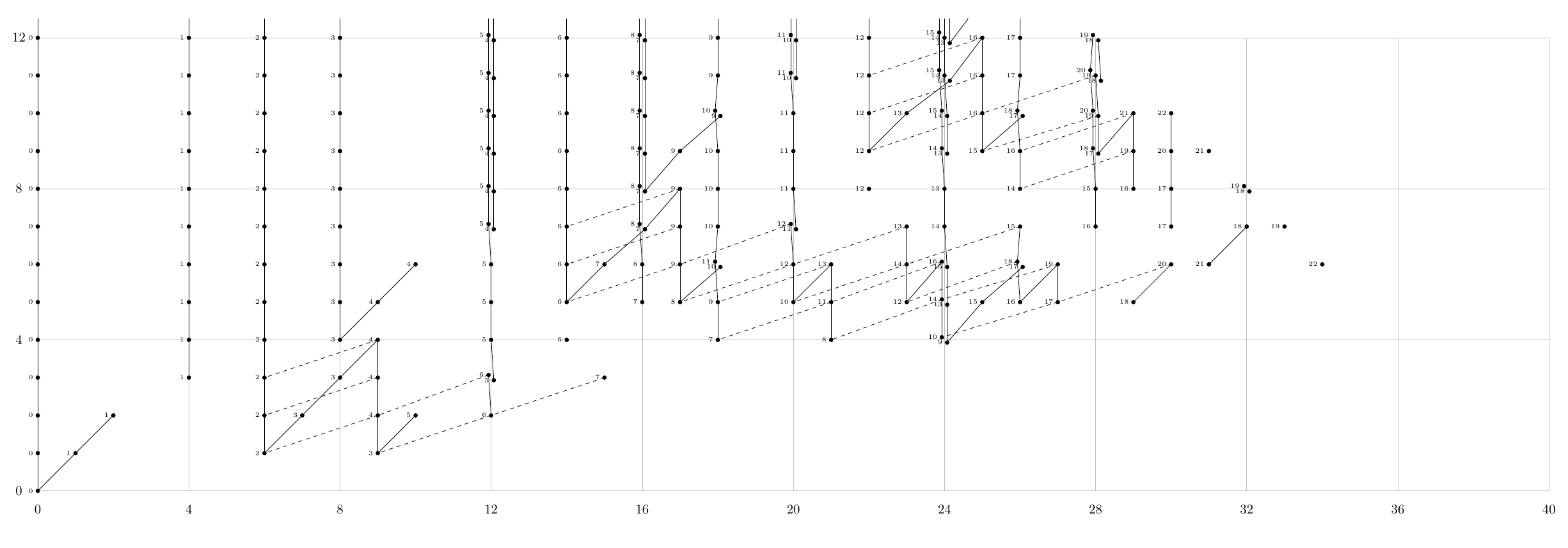}
    \caption{$E_2$-term of ASS for $\TJF_3$}
    \label{fig:TJF3ASS}
\end{figure}
\begin{figure}
    \centering
    \fullpagegraphic{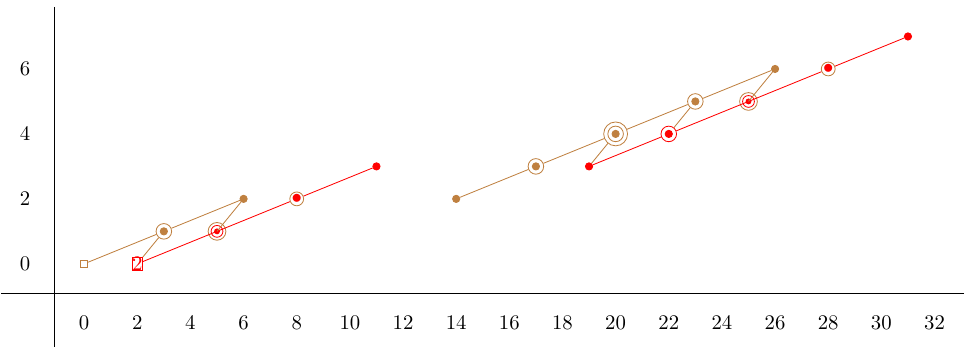}
    \caption{$E_2$-term of DSS for $\TMF \otimes C\eta$}
    \label{fig:E2Ceta}
\end{figure}
\subsection{Higher Differentials}
We now determine all differentials $d_r$ for $r \geq 5$. They are depicted in Figures~\ref{fig:TJF3higher1}–\ref{fig:TJF3higher4}.
\begin{lem}
    The $d_5$-differentials are:
    \begin{enumerate}
        \item $d_5([2\Delta^{1+2i} x_6]) = \nu [2
        \Delta^{2i} \bkappa x_6],$
        \item $d_5([\Delta^{1+2i} \nu^2 x_6]) = [\Delta^{2i} \bkappa \nu^3 x_6],$
        \item $d_5([2\Delta^{2+4i} x_6]) = 2 \nu [\Delta^{1+4i} \bkappa x_6] = [\Delta^{1+4i} \bkappa \delta x_4],$ 
        \item $d_5([\nu^2 \Delta^{2+4i} x_6]) = [\Delta^{1+4i} \eta \kappa x_0].$
    \end{enumerate}
\end{lem}
\begin{proof}
    These follow from $d_5(\Delta) = \nu \bkappa$ together with the Leibniz rule.
\end{proof}

\begin{lem}
    There are no $d_7$-differentials.
\end{lem}
\begin{proof}
    Degree considerations rule them out.
\end{proof}
\begin{lem}
    The $d_9$-differentials are:
    \begin{enumerate}
        \item $d_9([\Delta^2 x_0]) = [\bkappa^2 \nu x_4],$
        \item $d_9([\Delta^2 \nu x_4]) = [\bkappa^2 \kappa x_0]$
        \item $d_9([\Delta^3 x_0]) = [\Delta \bkappa^2 \nu x_4],$
        \item $d_9([\Delta^3 \nu x_4]) = [\Delta \bkappa^2 \kappa x_0].$
    \end{enumerate}
\end{lem}
\begin{proof}
    These are the images of the differentials for $\TJF_2 \to \TJF_3$.
\end{proof}
\begin{lem}
    The $d_{13}$-differentials are:
    \begin{enumerate}
        \item $d_{13}([\Delta^{2+4i} \eta x_0]) = [\Delta^{4i}\bkappa^3 \epsilon x_4],$
        \item $d_{13}([\Delta^5 \nu^3 x_6]) = [4 \Delta^2 \bkappa^4 x_6].$
    \end{enumerate}
\end{lem}
\begin{proof}
    The first is the image of $\TJF_2 \to \TJF_3$. The second is forced by the vanishing line of Lemma~\ref{lem:vanishing}.
\end{proof}

Using Lemma~\ref{lem:vanishing} and degree considerations, we obtain the remaining differentials.
\begin{lem}
    The $d_{15}$-differentials are:
    \begin{enumerate}
        \item $d_{15}([\Delta^{3+4i} \delta x_4]) = [\Delta^{4i}\bkappa^4 x_0].$
    \end{enumerate}
\end{lem}

\begin{lem}
    There are no $d_{17}$-differentials. The $d_{19}$-differentials are:
    \begin{enumerate}
        \item $d_{19}([\Delta^5 x_0]) = [\Delta \bkappa^{4} \nu^3 x_6]$
    \end{enumerate}
\end{lem}

\begin{lem}
    The $d_{21}$-differentials are:
    \begin{enumerate}
        \item $d_{21}([4 \Delta^6x_6]) = [\Delta^2 \bkappa^5 \eta x_0]$
    \end{enumerate}
\end{lem}

\begin{lem}
    The $d_{23}$-differentials are:
    \begin{enumerate}
        \item $d_{23}([\Delta^6 \eta x_0]) = [\bkappa^6 x_0].$
    \end{enumerate}
\end{lem}


\FloatBarrier

\section{\texorpdfstring{$\TJF_4$}{TJF4}}\label{sec:TJF4}
\subsection{\texorpdfstring{$E_2$-term}{E2-term}}
 The $E_1$-term of the algebraic AHSS computing the $E_2$-term of DSS is generated by $x_0, x_4, x_6,$ and $x_8$ with $d_1(x_8) = 2\nu x_4$. We also note the differential $d_3(\epsilon x_8) = \eta \kappa x_0$ in AHSS from the Massey product $\eta \kappa = \langle \epsilon, \nu, 2\nu \rangle.$ The resulting $E_2$-page of DSS is drawn in Figure \ref{fig:TJF4E2}. We now list the multiplicative extensions in the $E_2$-term.

\begin{prop}
    There are the following multiplicative extensions in the $E_2$-page:
    \begin{enumerate}
        \item $2 [\eta x_8] = [2\nu x_6]$,
        \item $2 [\nu^2 x_8] = [\kappa x_0],$
        \item $\eta [\eta \kappa x_8] = [4\bkappa x_4].$
    \end{enumerate}
\end{prop}
\begin{proof}
    \begin{enumerate}
        \item  
        The differential $d_1(x_8) = 2 \nu x_4$ gives
        \[
        [\delta x_4] = \langle \eta, \nu, 4x_4 \rangle =\eta \cdot [2 x_8] = 2 \cdot [\eta x_8].
        \]
        \item We have
        \[
        [\kappa x_0] = \toda{\nu, 2\nu, \nu, 2\nu} x_0 = \nu \toda{2\nu, \nu, 2\nu, x_0}.
        \]
        However, the class $\toda{\nu, 2\nu, x_0} = [2\nu x_4]$ is killed by $d_1(x_8) = 2\nu x_4$, so we compute
        \[
        \toda{2\nu, \nu, 2\nu, x_0} = [2\nu x_8]
        \]
        and obtain the desired result.
        \item Consider the element $[\eta \bkappa x_0] \in \pi_{21} \TJF_3$. The Toda bracket $\eta \bkappa = \toda{\kappa, 2, \nu^2}$ and the Moss' theorem implies the relation
        \begin{align*}
            [\eta \bkappa x_0] &= \toda{\kappa, 2, \nu^2} x_0 \\
            &= \kappa \toda{2, \nu^2, x_0} \\
            &= \kappa [2\nu x_4]
        \end{align*}
        in $\pi_{21} \TJF_3$. Therefore $[\eta \bkappa x_0]$ must be killed by a differential in the DSS of $\TJF_4$, and this extension is forced to kill this element via $d_3$.
    \end{enumerate}
\end{proof}
\subsection{Higher Differentials}

\begin{lem}
    $d_5$-differentials are determined by the following:
    \begin{enumerate}
        \item $d_5([2\Delta^{1+2i} x_6]) = \nu [2\Delta^{2i} \bkappa x_6],$
        \item $d_5([\Delta^{1+2i} \nu^2 x_6]) = [\Delta^{2i} \bkappa \nu^3 x_6],$
        \item $d_5([2\Delta^{2+4i} x_6]) = 2 \nu [\Delta^{1+4i} \bkappa x_6] = [\Delta^{1+4i} \bkappa \delta x_4],$ 
        \item $d_5([\nu^2 \Delta^{2+4i} x_6]) = [\Delta^{1+4i} \eta \kappa x_0],$
        \item $d_5([\Delta^{2+4i} \nu x_6]) = [\Delta^{1+4i} \bkappa \epsilon x_4],$
        \item $d_5([2\Delta^{1+4i} x_8]) = [2 \Delta^{4i} \nu \bkappa x_8],$
        \item $d_5([\Delta^{2+4i} \nu x_8]) = [\Delta^{1+4i} \bkappa \kappa x_0].$
    \end{enumerate}
\end{lem}
\begin{proof}
    Follows from $d(\Delta)$ and the Leibniz rule.
\end{proof}
\begin{lem}
    $d_7$-differentials are determined by the following:
    \begin{enumerate}
        \item $d_7([4\Delta^{1+4i} x_8]) = [\Delta^{4i}\eta^3 \bkappa x_8],$
        \item $d_7([2\Delta^{2+4i} x_8]) = [\Delta^{1+4i} \eta^2 \bkappa x_8]).$
    \end{enumerate}
\end{lem}
\begin{proof}
    This follows from Lemma~\ref{lem:naturalitytotarget}.
\end{proof}
\begin{lem}
    $d_9$-differentials are determined by the following:
    \begin{enumerate}
        \item $d_9([\Delta^{3+4i} \eta^2 x_8]) = [\Delta^{1+4i} \bkappa^2 \nu^3 x_8].$
    \end{enumerate}
\end{lem}
\begin{proof}
    This is a consequence of Lemma~\ref{lem:naturalitytotarget}.
\end{proof}
\begin{lem}
    There are the following $d_{11}$-differentials:
    \begin{enumerate}
        \item $d_{11}([\Delta^{2+4i} \nu \kappa x_6]) = [\Delta^{4i}\eta^2 \bkappa^3 x_8]$.
    \end{enumerate}
\end{lem}
\begin{proof}
    Since $[\Delta^{4i}\eta^2 \bkappa^3 x_8] = \eta \bkappa^3[\Delta^{4i} \eta x_8]$ and $\eta \bkappa^3 = 0$ in $\pi_{61}\TMF$, this class must be the target of a differential, and the one displayed is the only possibility.
\end{proof}
\begin{lem}
    There are the following $d_{13}$-differentials:
    \begin{enumerate}
        \item $d_{13}([\Delta^{3+4i} \nu^3 x_6]) = [\Delta^{4i} 2 \bkappa^4 x_6],$
        \item $d_{13}([\Delta^{3+4i}x_0]) = [2\Delta^{4i} \bkappa^3 \nu x_8],$
        \item $d_{13}([\Delta^{3+4i} \nu^3 x_8]) = [\Delta^{4i}2 \bkappa^4 x_8].$   
    \end{enumerate}
\end{lem}
\begin{proof}
    (1) is shown in $\TJF_3$. (2) follows from $[2\Delta^{4i} \bkappa^3 \nu x_8] = \bkappa^3\nu [2\Delta^{4i} x_8]$ and $\bkappa^3 \nu = 0$. (3) is deduced from Lemma~\ref{lem:naturalitytotarget}.
\end{proof}
\begin{lem}
    There are the following $d_{15}$-differentials:
    \begin{enumerate}
        \item $d_{15}([\Delta^{3+4i} \delta x_4]) = [\Delta^{4i}\bkappa^4 x_0],$
        \item $d_{15}([\Delta^4 \eta x_8]) = [\Delta \bkappa^4 x_0],$
        \item $d_{15}([\Delta^6 x_0]) = [\Delta^3 \bkappa^3 \eta^3 x_8].$
    \end{enumerate}
\end{lem}
\begin{proof}
    (1) is proven in $\TJF_3$. (2) and (3) follow from the vanishing line.
\end{proof}
Similarly, the following differentials come from Lemma~\ref{lem:vanishing}.
\begin{lem}
    There are the following $d_{17}$-differentials:
    \begin{enumerate}
        \item $d_{17}([\Delta^6 \eta^2 x_8]) = [\Delta^2 \bkappa^4 \nu \kappa x_8].$
    \end{enumerate}
\end{lem}
\begin{lem}
    The $d_{19}$-differentials are:
    \begin{enumerate}
        \item $d_{19}([\Delta^5 x_0]) = [\Delta \bkappa^{4} \nu^3 x_6],$
        \item $d_{19}([\Delta^5 \nu^3 x_6]) = [\Delta \bkappa^5 \eta^2 x_6].$
    \end{enumerate}
\end{lem}
\begin{lem}
    There are the following $d_{21}$-differentials:
    \begin{enumerate}
        \item $d_{21}([\Delta^5 \kappa \nu x_8]) = [\Delta^2 \bkappa^6 x_0].$
    \end{enumerate}
\end{lem}
\begin{lem}
    There are the following $d_{23}$-differentials:
    \begin{enumerate}
        \item $d_{23}([\Delta^5 \eta^2 x_8]) = [\bkappa^6 \eta x_8],$
        \item $d_{23}([\Delta^7 \eta^3 x_8]) = [\Delta^2 \bkappa^6 \eta^2 x_8].$
    \end{enumerate}
\end{lem}
\FloatBarrier

\section{\texorpdfstring{$\TJF_5$}{TJF5}}\label{sec:TJF5}
\subsection{\texorpdfstring{$E_2$-term}{E2-term}}
The differential in algebraic AHSS is given by $d_1(x_{10}) = \eta x_8 + \nu x_6$. The $E_2$-term of DSS for $\TJF_5$ is shown in Figure \ref{fig:TJF5E2}.

\begin{prop}
    There are the following multiplicative extensions in the $E_2$-page: 
    \begin{enumerate}
        \item $\eta [\nu^3 x_{10}] = [\bkappa x_0],$
        \item $2 [\eta \kappa x_{10}] = [\nu \kappa x_8],$
        \item $\eta [\eta \kappa x_{10}] = [2\bkappa x_6].$
    \end{enumerate}
\end{prop}
\begin{proof}
\begin{enumerate}
\item From the ASS for $\tmf \otimes P_6$ (Figure~\ref{fig:TJF6ASS}), the $\eta$-multiple of the $2$-torsion element in $\pi_{19} \tmf \otimes P_6$ is nonzero. Since both $[\nu^3 x_{10}]$ and $[\bkappa x_0]$ survive to the $E_2$-page of DSS for $\TJF_6$, this relation must hold in $\pi_* \TJF_6$, and filtration reasons ensure it already holds on the $E_2$-page.
\item The cofiber of $\TJF_3 \to \TJF_5$ is $\Sigma^8 \TMF \otimes C\eta$. The $E_2$-term of DSS for $\tmf \otimes C\eta$ (see \cite[Figure 4.3]{bauer_cpinfty}) exhibits the same multiplicative extension in $E_2^{3,28}$ for $\tmf \otimes \Sigma^8 C\eta$.
\item Similarly, the cofiber of $\TJF_3 \to \TJF_5$ is $\TMF \otimes \bC P^5_3$. The complex $\bC P^5_3$ has $6$-, $8$-, and $10$-cells attached by the stable map $\eta + \nu$; the extension in $\TMF \otimes \bC P^5_3$ follows exactly as in Section~\ref{sec:TJF2}.
\end{enumerate}
\end{proof}

\begin{figure}
    \centering
    \fullpagegraphic{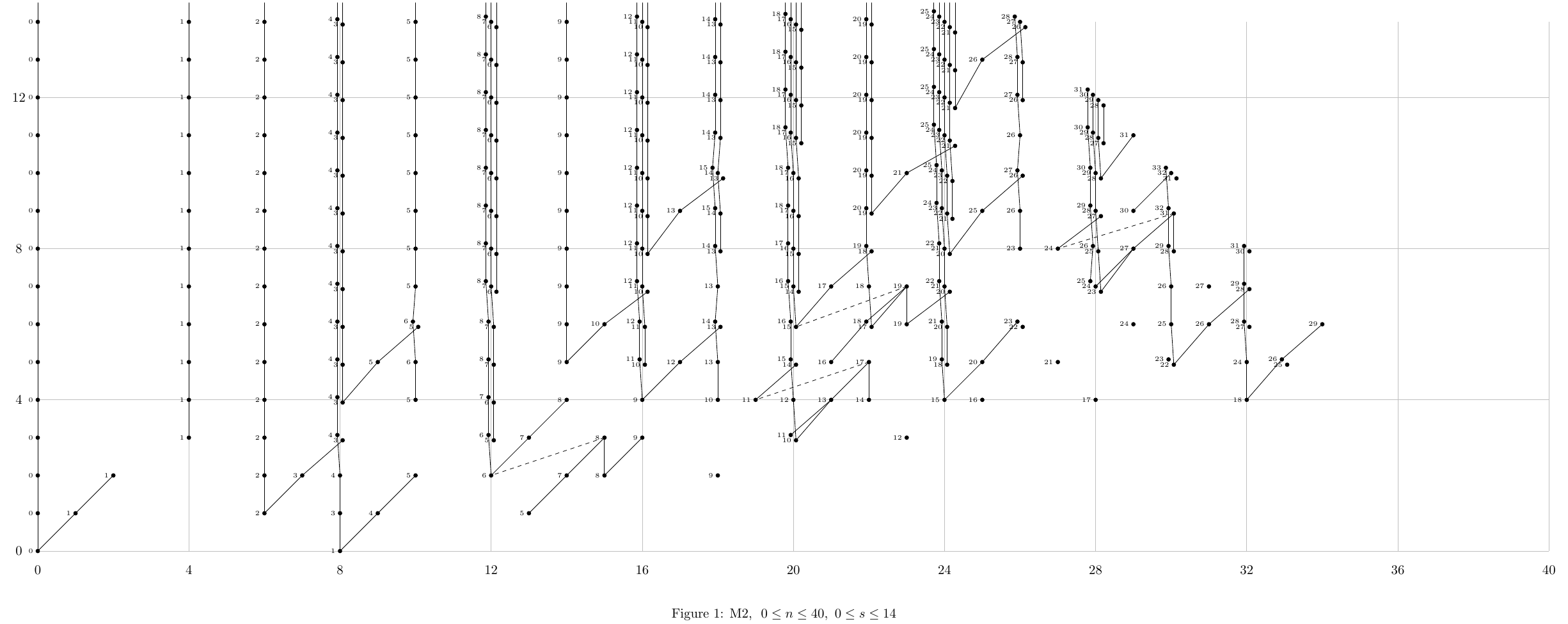}
    \caption{$E_2$-term of ASS for $\TJF_6$}
    \label{fig:TJF6ASS}
\end{figure}

\subsection{Higher Differentials}
The $d_3$-differentials follow from $d_3(\delta) = \eta^4$, and the resulting $E_4$-page is drawn in \ref{fig:TJF5higher1}. We determine differentials $d_r, r\geq 5$ below.
\begin{lem}
    There are the following $d_5$-differentials:
    \begin{enumerate}
        \item $d_5([2\Delta^{1+4i} x_8]) = [2 \Delta^{4i} \nu \bkappa x_8],$
        \item $d_5([\Delta^{2+4i} \nu x_8]) = [\Delta^{1+4i} \bkappa \kappa x_0].$
    \end{enumerate}
\end{lem}

\begin{proof}
    Verified by the Leibniz rule.
\end{proof}

\begin{lem}
    There are no $d_7$-differentials, and there are the following $d_9$-differentials:
    \begin{enumerate}
        \item $d_9([\Delta^{2+4i} \nu^3 x_{10}]) = [2\Delta^{4i}\bkappa^3 x_6],$
        \item $d_9([\Delta^{3+4i} \nu^3 x_{10}]) = [2\Delta^{1+4i} \bkappa^3 x_6],$
        \item $d_9([\Delta^{2+4i} x_6]) = [\Delta^{4i} \bkappa^2 \delta x_8],$
        \item $d_9([\Delta^{3+4i} x_6]) = [\Delta^{1+4i} \bkappa^2 \delta x_8].$
    \end{enumerate}
\end{lem}

\begin{proof}
    Note $[2\bkappa^4 x_6] = \bkappa^3 [2\bkappa x_6] = \bkappa^3 \eta[\eta \kappa x_6]$ in the $E_2$-term. Because $\bkappa^3 \eta = 0$ in $\pi_* \tmf$, some differential must hit $[2\bkappa^4 x_6]$; (1) and (2) are forced by this and linearity.

    (3) and (4) follow from Lemma~\ref{lem:vanishing}.
\end{proof}
\begin{lem}
    There are the following $d_{11}$-differentials:
    \begin{enumerate}
        \item $d_{11}([\Delta^{2+4i} \delta x_8]) = [\Delta^{4i} \bkappa^3 x_0],$
        \item $d_{11}([\Delta^{3+4i} \delta x_8]) = [\Delta^{1+4i} \bkappa^3 x_0].$
    \end{enumerate}
\end{lem}
\begin{proof}
    Because $[\bkappa^4 x_0] = 0$ in $\pi_{80} \TJF_4$, some differential in the DSS for $\TJF_5$ must hit $[\bkappa^4 x_0]$, and the displayed $d_{11}$ is the only possibility. 
\end{proof}
\begin{lem}
    There are the following $d_{13}$-differentials:
    \begin{enumerate}
        \item $d_{13}([\Delta^{3+4i}x_0]) = [\Delta^{4i} \bkappa^3 \nu x_8],$
        \item $d_{13}([\Delta^{3+4i} \nu^2 x_8]) = [\bkappa^{3+4i} \eta \kappa x_{10}],$
        \item $d_{13}([\Delta^5 \nu^2 x_8]) = [\Delta^2 \bkappa^3 \nu \kappa x_8]$
    \end{enumerate}
\end{lem}

\begin{proof}
    These are the images of the $d_{13}$-differentials in $\TJF_4$.
\end{proof}

The vanishing-line Lemma~\ref{lem:vanishing} yields the following long differentials.
\begin{lem}
    There are the following $d_{15}$-differentials:
    \begin{enumerate}
        \item $d_{15}([\Delta^3 x_6]) = [\bkappa^3 \nu^3 x_{10}]$
    \end{enumerate}
\end{lem}
\begin{lem}
    There are no $d_{17}$-differentials, and there are the following $d_{19}$-differentials:
    \begin{enumerate}
        \item $d_{19} ([\Delta^5 \kappa \eta x_{10}]) = [\Delta \bkappa^5 \nu^2 x_8]$.
    \end{enumerate}
\end{lem}

\begin{lem}
    There are the following $d_{21}$-differentials:
    \begin{enumerate}
        \item $d_{21}([\Delta^6 \eta \kappa x_{10}]) = [\Delta^2 \bkappa^6 x_0]$
    \end{enumerate}
\end{lem}
\begin{lem}
    There are the following $d_{23}$-differentials:
    \begin{enumerate}
        \item $d_{23}([\Delta^6 x_0]) = [\bkappa^5 \nu^3 x_{10}].$
    \end{enumerate}
\end{lem}
\FloatBarrier

\section{\texorpdfstring{$\TJF_6$}{TJF6}}\label{sec:TJF6}
\subsection{\texorpdfstring{$E_2$-term}{E2-term}}
The $E_2$-term of DSS for $\TJF_6$ is shown in Figure \ref{fig:TJF6E2}.

\begin{prop}\label{prop:multE2TJF6}
    There are the following multiplicative extensions in the $E_2$-page:
    \begin{enumerate}
        \item $2[\eta x_{12}] = \eta[4 x_{12}] = [\delta x_8],$
        \item $\eta [\eta \kappa x_{12}] = [2\bkappa x_8].$
    \end{enumerate}
\end{prop}
\begin{proof}
    \begin{enumerate}
        \item First, we see $[\delta x_8] = \toda{\eta, 2\nu, 2x_8} = \eta[4x_{12}]$. Next we show $2[\eta x_{12}] \neq 0$: we know that $E_2^{1,14}$ of DSS for $\tmf \otimes P_6$ has order $8$ generated by elements $[\delta x_8], [\eta x_{12}], $ and $[c_6 x_0]$ (the last term is omitted in Figure \ref{fig:TJF6E2} for readability). The composition with $\tmf \to \tmf \otimes P_6$ and $\tmf \otimes P_6 \to \tmf \otimes P_\infty$ induces an isomorphism on $E_2^{1,14}$, so we see that $[c_6 x_0]$ generates the $\bZ / 2$ summand. Therefore $E_{2}^{1, 14}$ is isomorphic to either $\bZ / 4 \oplus \bZ / 2$ or $\bZ/ 2 \oplus \bZ /2 \oplus \bZ /2$.  However, we know that $E_2^{1,14}$ in DSS for $\TJF_7$ is isomorphic to the one in $\TJF_\infty$ because the elements of degree less than $15$ in the cobar complex $(B, \Sigma)$ are included in the subcomplex $(B_7, \Sigma_7)$ in section \ref{sec:E2term}. Thus $E_2^{1,14}$ on $\tmf \otimes P_7$ is isomorphic to $\bZ / 2$ generated by $[c_6 x_0]$ and the exact sequence induced from the cofiber sequence $\Sigma^{13} \tmf \to \tmf \otimes P_6 \to \tmf \otimes P_7$ becomes $\bZ \to E_2 \to \bZ / 2$, and therefore $E_2^{1,14}$ on $\tmf \otimes P_6$ must be isomorphic to $\bZ / 4 \oplus \bZ /2$. This group structure forces the equation $2[\eta x_{12}] = \eta[4x_{12}]$ because the latter is the nonzero $2$-torsion element in $\bZ / 4$.
        \item Note that the cofiber of the map $\TJF_4 \to \TJF_6$ is equivalent to $\TMF \otimes X$, where $X$ has $8$-cell, $10$-cell, and $12$-cell with the attaching map $\eta + \nu$ from the $12$-cell. Then this extension is shown in the same way as proposition \ref{E2extTJF2}, (5).
    \end{enumerate}
\end{proof}
\begin{rem}
    Note that the element in coordinate $(18,2)$ is represented by $[\nu^2 x_{12} + \epsilon x_{10}]$, so we have an extension $\eta [\nu^2 x_{12} + \epsilon x_{10}] = [\eta \epsilon x_{10}]$.
\end{rem}
\subsection{Higher Differentials}
\begin{lem}
    There are the following $d_5$-differentials:
    \begin{enumerate}
        \item $d_5([\Delta^{2i+1} \eta \kappa x_{10}]) = [\Delta^{2i}\bkappa^2 x_8],$
        \item $d_5([\Delta^{2i+1} \nu^2 x_{12}]) = [\Delta^{2i} \bkappa \nu^3 x_{12}].$
    \end{enumerate}        
\end{lem}
\begin{proof}
    It can be verified from the Leibniz rule.
\end{proof}
\begin{lem}
    There are the following $d_7$-differentials:
    \begin{enumerate}
        \item $d_7([4\Delta^{2i + 1} x_{12}]) = [\Delta^{2i} \eta^3 \bkappa x_{12}].$
    \end{enumerate}
\end{lem}
\begin{proof}
    This is a consequence of \ref{lem:naturalitytotarget}.
\end{proof}
\begin{lem}
    There are the following $d_9$-differentials:
    \begin{enumerate}
        \item $d_9([\Delta^{2+4i} x_6]) = [\Delta^{4i} \bkappa^2 \delta x_8],$
        \item $d_9([\Delta^{3+4i} x_6]) = [\Delta^{1+4i} \bkappa^2 \delta x_8],$
        \item $d_9([\Delta^{2+4i} \eta x_{12}]) = [\Delta^{4i} \bkappa^3 x_{12}],$ 
        \item $d_9([\Delta^{3+4i} \eta x_{12}]) = [\Delta^{1+4i} \bkappa^3 x_{12}].$
    \end{enumerate} 
\end{lem}
\begin{proof}
    (1) and (2) are shown in $\TJF_6$. (3) and (4) can be deduced from \ref{lem:naturalitytotarget}.
\end{proof}

\begin{lem}
    There are the following $d_{11}$-differentials:
    \begin{enumerate}
        \item $d_{11}([\Delta^{2+4i} \delta x_8]) = [\Delta^{4i} \bkappa^3 x_0],$
        \item $d_{11}([\Delta^{3+4i} \delta x_8]) = [\Delta^{1+4i} \bkappa^3 x_0].$
        \item $d_{11}([2 \Delta^{3+4i} x_8]) = [\Delta^{1+4i} \bkappa^2 \eta^3 x_{12}],$
        \item $d_{11}([\Delta^{2+4i} \kappa \eta x_{12}]) = [\bkappa^2\eta^2 x_{12}],$
        \item $d_{11}([\Delta^5 \eta x_{12}]) = [\Delta^3 \bkappa^3 x_0].$
    \end{enumerate}
\end{lem}

\begin{proof}
    (1) and (2) are shown in $\TJF_6$. (3) and (5) can be verified by \ref{lem:vanishing}. (4) is a consequence of \ref{lem:naturalitytotarget}.
\end{proof}

\begin{lem}
    There are no $d_{13}$ or $d_{15}$-differentials.
\end{lem}

\begin{proof}
    This can be checked from the degree reason.
\end{proof}

Using the lemma \ref{lem:vanishing}, you can verify the differentials below.
\begin{lem}
    There are the following $d_{19}$-differentials:
    \begin{enumerate}
        \item $d_{19}([\Delta^5 \nu^2 x_{12}]) = [\bkappa^5 \eta x_{12}],$
        \item $d_{19}([\Delta^7 x_0]) = [\Delta^3 \bkappa^4 \eta^3 x_{12}].$
    \end{enumerate}
\end{lem}
\begin{lem}
    There are no $d_{21}$-differentials, and there are the following $d_{23}$-differentials:
    \begin{enumerate}
        \item $d_{23}([\Delta^5 \nu^3 x_{10}]) = [\bkappa^6 \nu^2 x_{12}],$
        \item $d_{23}([\Delta^6 x_0]) = [\bkappa^5 \nu^3 x_{10}],$
        \item $d_{23}([\Delta^6 \eta^2 x_{12}]) = [\Delta \bkappa^6 \eta x_{12}],$
        \item $d_{23}([\Delta^7 \eta^3 x_{12}]) = [\Delta^2 \bkappa^6 \eta^2 x_{12}].$
    \end{enumerate}
\end{lem}

\FloatBarrier
\section{\texorpdfstring{$\TJF_7$}{TJF7}}\label{sec:TJF7}
\subsection{\texorpdfstring{$E_2$-term}{E2-term}}
The $E_2$-term is shown in Figure \ref{fig:TJF7E2}.

\begin{prop}
    There are the following multiplicative extensions in the $E_2$-page:
    \begin{enumerate}
        \item $4[\nu x_{14}] = [\delta x_{12}],$
        \item $\eta [\nu x_{14}] = [\nu^2 x_{12} + \epsilon x_{10}],$
        \item $2[\nu^2 x_{14}] = [\epsilon x_{12}],$
        \item $\nu[\nu^3 x_{14}] = [2\bkappa x_{6}],$
        \item $\eta [\kappa \nu x_{14}] = [4\bkappa x_{12}].$
    \end{enumerate}
\end{prop}
\begin{proof}
    Note that the cofiber of the inclusion $\TJF_5 \to \TJF_7$ is equivalent to $\Sigma^{12} \TMF \otimes C\eta$. Comparing the $E_2$-term of DSS for $\tmf \otimes C\eta$ in Figure \ref{fig:E2Ceta}, extensions (1), (2), (3), and (5) follow. Extension (4) is shown similarly to the earlier proposition \ref{prop:multE2TJF6} (1): the structure in $\TJF_\infty$ forces this multiplicative extension.
\end{proof}
\subsection{Higher Differentials}
$d_3$-differentials and $d_5$-differentials are determined by the Leibniz rule together with $d_3(\delta) = \eta^4$ and $d_5(\Delta) = \nu \bkappa$. 
\begin{lem}
    There are the following $d_5$-differentials:
    \begin{enumerate}
        \item $d_5([4\Delta x_{14}]) = [\bkappa \delta x_{12}],$
        \item $d_5([\Delta \nu x_{14}]) = [\bkappa \nu^2 x_{14}],$
        \item $d_5 ([\Delta^2 \nu x_{14}]) = [\Delta \bkappa \epsilon x_{12}].$
    \end{enumerate}
\end{lem}
\begin{lem}
    There are the following $d_7$-differentials:
    \begin{enumerate}
        \item $d_7([\Delta^{2i+1} \delta x_{12}]) = [\Delta^{2i}\bkappa^2 x_0],$
        \item $d_7([\Delta^{4i +2} 2\nu x_{14}]) = [\Delta^{4i + 1} \bkappa^2 x_0],$
        \item $d_7([\Delta^4 \nu x_{14}]) = [\Delta^3 \bkappa^2 x_0].$
    \end{enumerate}
\end{lem}
\begin{proof}
    (1) is shown in $\TJF_6$. (2) can be shown by the synthetic Leibitz rule and the relation $4\nu = \tau^2 \eta^3$ in $\syn{\tmf}$. (3) is the consequence of $d_7(\Delta^4) = \Delta^3 \eta^3 \bkappa$ in the DSS of $\tmf$.
\end{proof}
\begin{lem}
    There are the following $d_9$-differentials:
    \begin{enumerate}
        \item $d_9([\Delta^{2+4i} \delta \eta x_{12}]) = [\Delta^{4i}\bkappa^2 \kappa \eta x_{10}],$
        \item $d_9([\Delta^{2+4i} \kappa \eta x_{10}]) = [4\Delta^{4i}\bkappa^3 x_{12}],$
        \item $d_9([\Delta^{3+4i} \delta \eta x_{12}]) = [\Delta^{1 + 4i}\bkappa^2 \kappa \eta x_{10}],$
        \item $d_9([\Delta^{3 +4i} \kappa \eta x_{10}]) = [4\Delta^{1 + 4i}\bkappa^3 x_{12}].$
    \end{enumerate}
\end{lem}
\begin{proof}
    All these differentials are images of $d_9$ under the map $\TJF_6 \to \TJF_7$.
\end{proof}

\begin{lem}
    There are the following $d_{11}$-differentials:
    \begin{enumerate}
        \item $d_{11} ([4 \Delta^{2+4i} x_{12}]) = [\Delta^{4i} \bkappa^2 \nu^3 x_{10}],$
        \item $d_{11} ([4 \Delta^{3+4i} x_{12}]) = [\Delta^{1 + 4i} \bkappa^2 \nu^3 x_{10}].$
    \end{enumerate}
\end{lem}
\begin{proof}
    Again, those differentials are images of $d_{11}$ under the map $\TJF_6 \to \TJF_7.$
\end{proof}
\begin{lem}
    There are the following $d_{13}$-differentials:
    \begin{enumerate}
        \item $d_{13}([2 \Delta^{3+4i} x_8]) = [2\Delta^{4i}\bkappa^3 \nu x_{14}],$
        \item $d_{13}([\Delta^{3+4i} \nu^2 x_{14}]) = [\Delta^{4i} \bkappa^3 \kappa \nu x_{12}].$
    \end{enumerate}
\end{lem}
\begin{proof}
These can be verified from the vanishing line \ref{lem:vanishing}.
\end{proof}

\begin{lem}
    There are no $d_{15}, d_{17}, d_{19}, d_{21}$-differentials.
\end{lem}
\begin{proof}
Follows from the degree reason.
\end{proof}
\begin{lem}
    There are the following $d_{23}$-differentials:
    \begin{enumerate}
        \item $d_{23}([\Delta^5 \delta \eta x_{12}]) = [\bkappa^6 \nu x_{14}],$
        \item $d_{23}([\Delta^6 \delta \eta x_{12}]) = [\Delta \bkappa^6 \delta \eta x_{12}],$
        \item $d_{23}([\Delta^7 x_0]) = [\Delta^2 \bkappa^5 \nu^3 x_{10}].$
    \end{enumerate}
\end{lem}


\appendix
\section{A comparison between algebraic and analytic Jacobi forms}\label{App:JacobiForms}

\begin{defn}
    The Tate curve is the projective plane curve over $\bZ[[q]]$ given by the equation
    \[
    y^2 + xy = x^3 + a_4(q) x + a_6(q)
    \]
    where $a_4(q)$ and $a_6(q)$ are defined by
    \[
        a_4(q) = -5\sum_{n \geq 1} \frac{n^3 q^n}{1 - q^n}
    \]
    and
    \[
        a_6(q) = -\sum_{n \geq 1} \frac{q^n(7n^5 + 5 n ^3)}{12(1-q^n)}.
    \]
    The Tate curve becomes a smooth elliptic curve over $\bZ((q))$, denoted $E_q$. It has the canonical invariant differential $\omega_{\can} = dx/(2x+y)$.
\end{defn}
\begin{defn}
    Let $R$ be a ring, and let $E_R \to \cM_{ell}(R)$ be the universal elliptic curve over the moduli of smooth elliptic curves over $R$. For any weight $k \in \bZ$ and index $m \in \frac{\bZ}{2}$, the weak Jacobi forms with coefficients in $R$ are the global sections $\Gamma(E_R, p^* \omega^{k+2m} \otimes \cO_{E_R} (2me)).$

    Equivalently, a Jacobi form of weight $k$ and index $m$ with coefficients in $R$ assigns to each pair $(C/\Spec R', \widetilde{\omega})$, an elliptic curve $C$ over an $R$-algebra $R'$ together with a trivialization $\widetilde{\omega}$ of the invariant differential, an element
    \[
    f(C/R', \widetilde{\omega}) \in \Gamma(C, p^* \omega^{2m} \otimes\cO_{C}(2me))
    \]
    subject to the following:
    \begin{itemize}
        \item $f$ commutes with the base change, and
        \item if $\lambda \in R'^\times$, then
        \[
        f(C/R', \lambda \widetilde{\omega}) = \lambda^{-k} f(C/R', \widetilde{\omega}).
        \]
    \end{itemize}
\end{defn}
Evaluating an arithmetic Jacobi form $f$ on the Tate curve $C_\Tate / \bZ((q))$ with the canonical generator $\omega_\can$ yields an element $f(C_\Tate / \bZ((q)) , \omega_\can) \in \Gamma(C_\Tate, \cO_{C_\Tate}(2me))$. Because $C_\Tate \cong \bG_m / q^\bZ$ \cite{DeligneRapoport}, any section in $\Gamma(C_\Tate, \cO_{C_\Tate}(2me))$ can be written as $s \in \bZ((q))[u^{\pm1}]$ satisfying
\[
s(qu) = q^{-m}u^{-2m}s(u).
\]
For a weak Jacobi form $f$, the resulting two-variable power series $f(C_\Tate, \omega_\can)$ is its $q$-expansion.
\begin{prop}[$q$-expansion principle]
    If a weak Jacobi form $f$ over a ring $R$ has $q$-expansion coefficients in a subring $S \subset R$, then there exists a unique weak Jacobi form $\tilde{f}$ over $S$ whose base change is $f$. In particular, a weak Jacobi form over $R$ is uniquely determined by its $q$-expansion.
\end{prop}
\begin{proof}
    This is proven in \cite[Lemma 2.6]{Kramer1995}.
\end{proof}

We now consider $q$-expansions over $\bC$. The moduli of complex elliptic curves is the quotient $[\bH / \SL_2(\bZ)]$, and an elliptic curve over $\bC$ can be written as $\bC / (\bZ \tau + \bZ)$ for some $\tau \in \bH$. The universal complex elliptic curve over $\cM_{ell}(\bC)$ can therefore be written as
\[
E_\bC \cong [\bH \times \bC / \SL_2(\bZ) \ltimes \bZ^2]
\]
where $\SL_2(\bZ)$ acts on $\bZ^2$ by matrix multiplication. The action of $(\gamma, (m,n)) \in \SL_2(\bZ) \ltimes \bZ^2$ for $\gamma = \begin{pmatrix}
   a & b \\
   c & d\end{pmatrix}$ is given by
\begin{align*}
    (\gamma, (\lambda,\mu)) \cdot (\tau, z) &= \left( \gamma\cdot \tau, \frac{z + \lambda \tau+\mu}{c \tau + d} \right) \\
    &= \left( \frac{a\tau +b}{c\tau + d}, \frac{z + \lambda\tau+\mu}{c \tau + d} \right). \\
\end{align*}
\begin{prop}
    Let $f$ be a weak Jacobi form over $\bC$, and write its $q$-expansion as $\phi(\tau, z) \coloneqq f(E_q /\bC((q)), \omega_\can)$. For any $\left( \begin{pmatrix}
   a & b \\
   c & d\end{pmatrix}, (\lambda, \mu)\right) \in \SL_2(\bZ) \ltimes \bZ^2$, we have
    \[
    \phi \left( \frac{a\tau + b}{c\tau + d}, \frac{z + \lambda \tau + \mu}{c\tau + d} \right) = (c\tau + d)^{k} e^{-2\pi i m \left(\lambda^2 \tau + 2 \lambda z - \frac{c(z+\lambda\tau + \mu)^2}{c\tau +d}\right)} \phi(\tau,z).
    \]
    Thus $\phi(\tau, z)$ is an analytic Jacobi form of weight $k$ and index $m$.
\end{prop}
\begin{proof}
    Recall the isomorphism (\cite[Appendix I.2]{mumford_abelian_2008})
    \[
    H^1(\SL_2(\bZ) \ltimes \bZ^2, \bC^\times) \cong H^1(E, \cO_{E}^\times) \cong \Pic(E).
    \]
    One checks that
    \[
    \left( \begin{pmatrix}
   a & b \\
   c & d\end{pmatrix}, (\lambda, \mu)\right) \mapsto e^{2\pi i m (\lambda^2 \tau + 2 \lambda z - \frac{c(z+\lambda\tau + \mu)^2}{c\tau +d})}
    \]
    defines a $1$-cocycle in $H^1(\SL_2(\bZ) \ltimes \bZ^2, \bC^\times)$. The corresponding line bundle is $p^* \omega^{2m} \otimes \cO_E (2me)$ by the Appel–Humbert theorem (see \cite[Section 2]{Kramer1991}). Hence global sections $\Gamma(E_\bC, p^* \omega^{k+2m} \otimes\cO_{E_\bC}(2me))$ are functions $\phi \colon \bH \times \bC \to \bC$ satisfying the Jacobi form transformation law.
\end{proof}
Evaluating at the Tate curve with $\omega_{\can}$ gives an injective map
\[
\Gamma(E_\bC, p^* \omega^{k} \otimes \cO_{E_\bC
} (2me)) \to \JF_{k, m}.
\]

\begin{thm}[\cite{gritsenko_graded_2020}, Corollary 3.6]
    The graded ring of analytic Jacobi forms is generated by
    \[
    \bigoplus_{n,k} \JF_{n,k} \cong \bZ[c_4, c_6, \Delta^{\pm 1}, E_{4,1}, E_{4,2}, E_{4,3}, E_{6,1}, E_{6,2}, F_{6,3}, \phi_{0,1}, \phi_{0,2}, \phi_{0, \frac32}, \phi_{0,4}, \phi_{-1,\frac12}]
    \]
    with relations
    \[
    \begin{array}{ll}
        c_4^3 - c_6^2 = 1728 \Delta, &
        24 \phi_{0,2} = \phi_{0,1} - c_4 \phi_{-1, \frac12}^4, \\
        432 \phi_{0,\frac32}^2 = \phi_{0,1}^3 - 3c_4 \phi_{0,1}\phi_{-1, \frac12}^4 + 2c_6\phi_{-1, \frac12}^6, &
        4\phi_{0,4} = \phi_{0,1}\phi_{0, \frac32}^2 - \phi_{0,2}^2, \\
        12E_{4,1} = c_4\phi_{0,1} - c_6\phi_{-1, \frac12}^2, &
        12E_{6,1} = c_6 \phi_{0,1} - c_4^2\phi_{-1, \frac12}^2, \\
        6E_{4,2} = E_{4,1}\phi_{0,1} - c_4\phi_{0,2}, &
        6E_{6,2} = E_{6,1} \phi_{0,1} - c_6 \phi_{0,2}, \\
        2E_{4,3} = E_{4,1}\phi_{0,2} - c_4 \phi_{0,3}, &
        2F_{6,3} = E_{6,1}\phi_{0,2} - c_6\phi_{0,3}.
    \end{array}
    \]
\end{thm}

\section{Descent spectral sequence charts}
We display the descent spectral sequences for $\TJF_n$: each chart shows the DSS for $\tmf \otimes P_n$, and the DSS for $\TJF_n$ is obtained by inverting $\Delta$. Throughout we use Adams indexing: $E_r^{s,t}$ is plotted at $(t-s, s)$.
\begin{itemize}
    \item A square indicates the $\bZ_{(2)}$-summand.
    \item A number $n$ in a square indicates the $n \bZ_{(2)}$-summand.
    \item A dot indicates the $\bZ / 2$-summand.
    \item A $n$ concentric circle indicates $\bZ / 2^n$-summand. For example, $\circledcirc$ indicates the $\bZ/4$-summand.
    \item Colors indicate the cellular filtration of the corresponding element in its $E_2$-term from Section \ref{sec:E2term}. 
    \begin{center}
    \begin{tabular}{c|c|c|c|c|c|c|c}
        & $x_0$ & $x_4$ & $x_6$ & $x_8$ & $x_{10}$ & $x_{12}$ & $x_{14}$ \\ \hline
        colors & black & brown & red & orange & purple & green & blue
    \end{tabular}
    \end{center}
    \item Non-vertical arrows with negative slope indicate differentials.
    \item Lines with non-negative slope indicate multiplication by an element of $\pi_* \TMF$: vertical for $\times 2$, slope $1$ for $\times \eta$, slope $1/3$ for $\times \nu$.
    \item Line colors record the image of the generator of differential or multiplication: e.g. in Figure \ref{fig:TJF3E2}, $\nu$ times the generator in $E_2^{0,6}$ is $[2\nu x_6]$ (red). The brown line from $(9,1)$ to $(12,2)$ likewise indicates $\nu$ times the generator in $E_2^{1,10}$ is $[\epsilon x_4]$ (brown).
    \item A diamond $\diamond$ abbreviates an $\eta$-tower, meaning all $\eta$-multiples from that point survive.
\end{itemize}

\subsection{\texorpdfstring{$\TJF_2$}{TJF2}}
This section collects diagrams of the descent spectral sequence for $\TJF_2$.

\par
\begingroup
\centering

\vspace*{\fill} 

\widegraphic{TJF2E2.pdf}\par
\captionof{figure}{The $E_{2}$-page of DSS for $\TJF_{2}$}
\label{fig:TJF2E2app}

\vfill 

\widegraphic{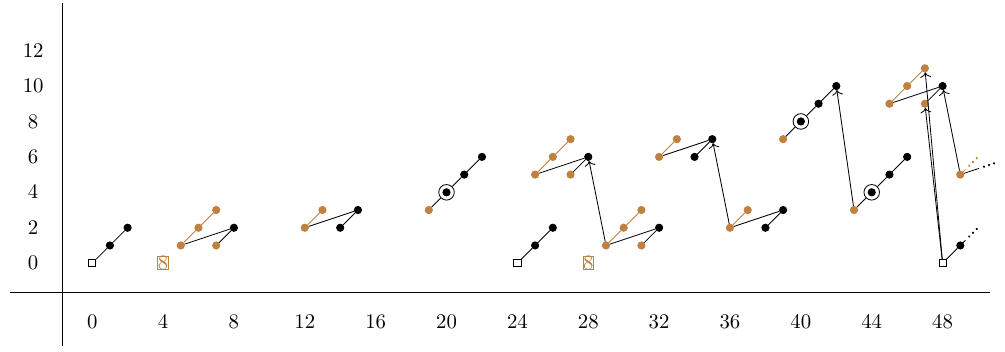}\par
\captionof{figure}{The higher differentials of DSS for $\TJF_{2}$ in $0-49$ range}
\label{fig:TJF2higher1}

\vspace*{\fill}
\par
\endgroup

\begin{figure}[hp]
\centering
    \verticalgraphic{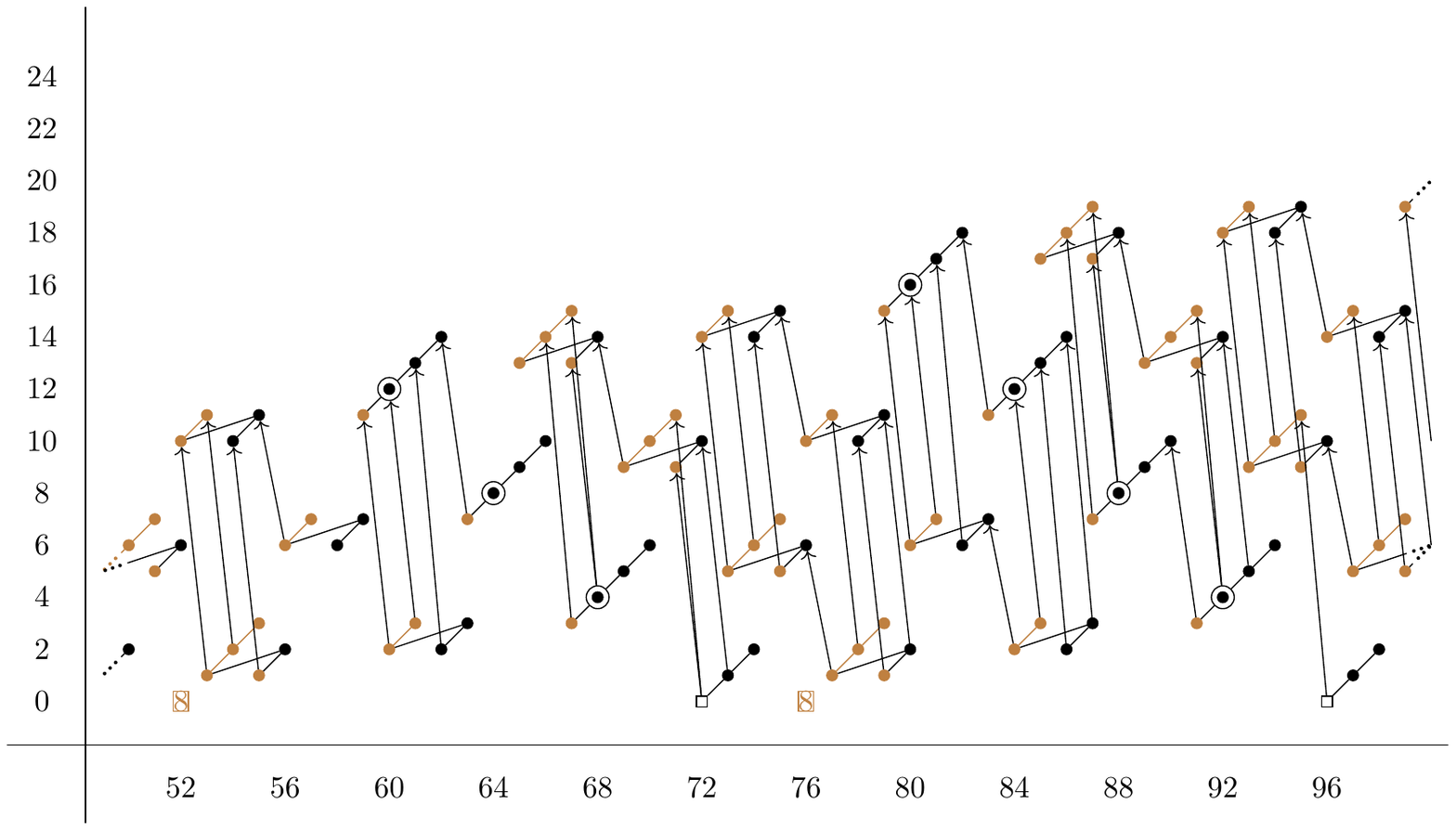}
    \caption{The higher pages of DSS for $\TJF_{2}$ and differentials}
    \label{fig:TJF2higher2}
\end{figure}
\begin{figure}[hp]
\centering
    \verticalgraphic{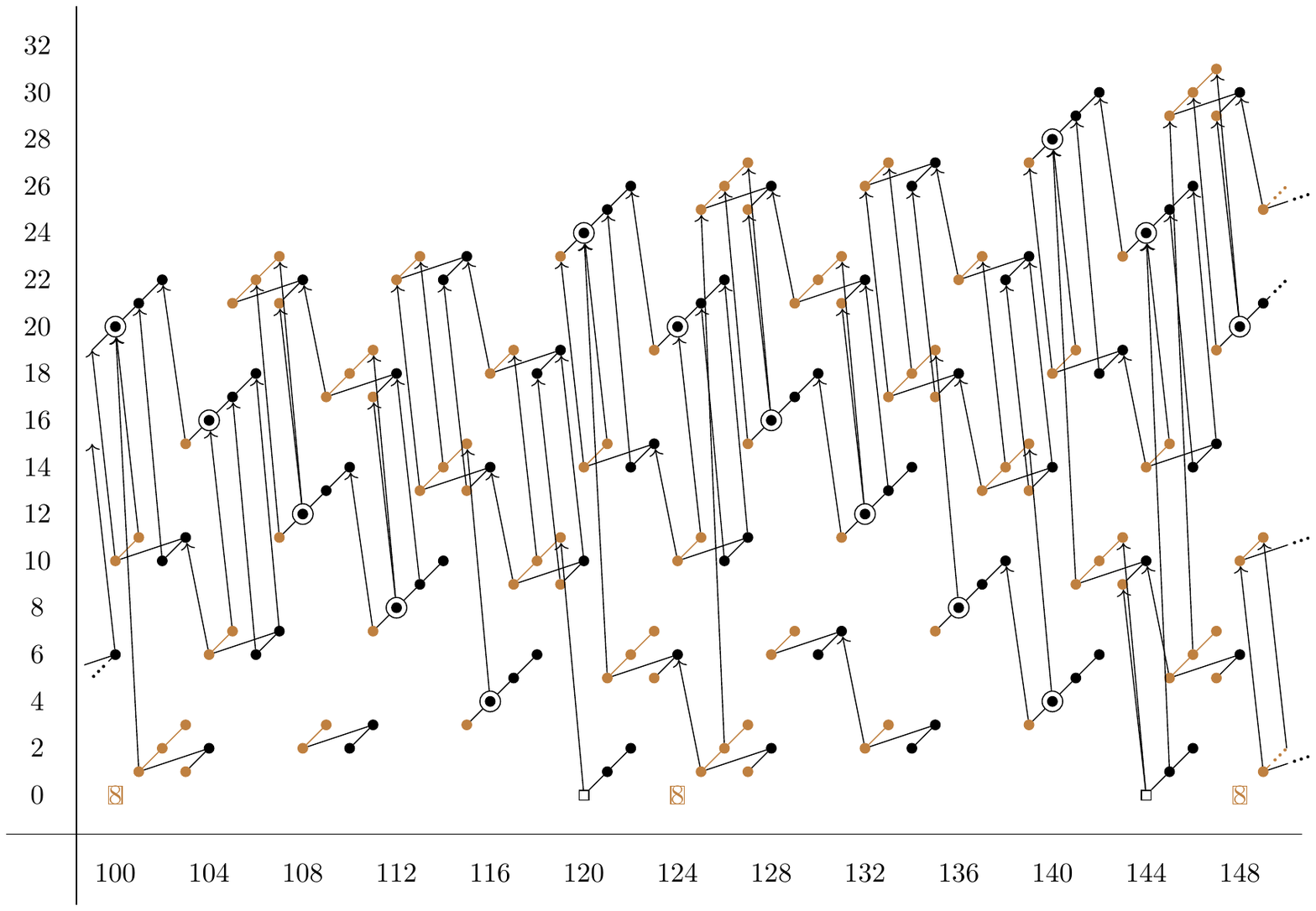}
    \caption{The higher pages of DSS for $\TJF_{2}$ and differentials}
    \label{fig:TJF2higher3}
\end{figure}

\begin{figure}[hp]
\centering
    \verticalgraphic{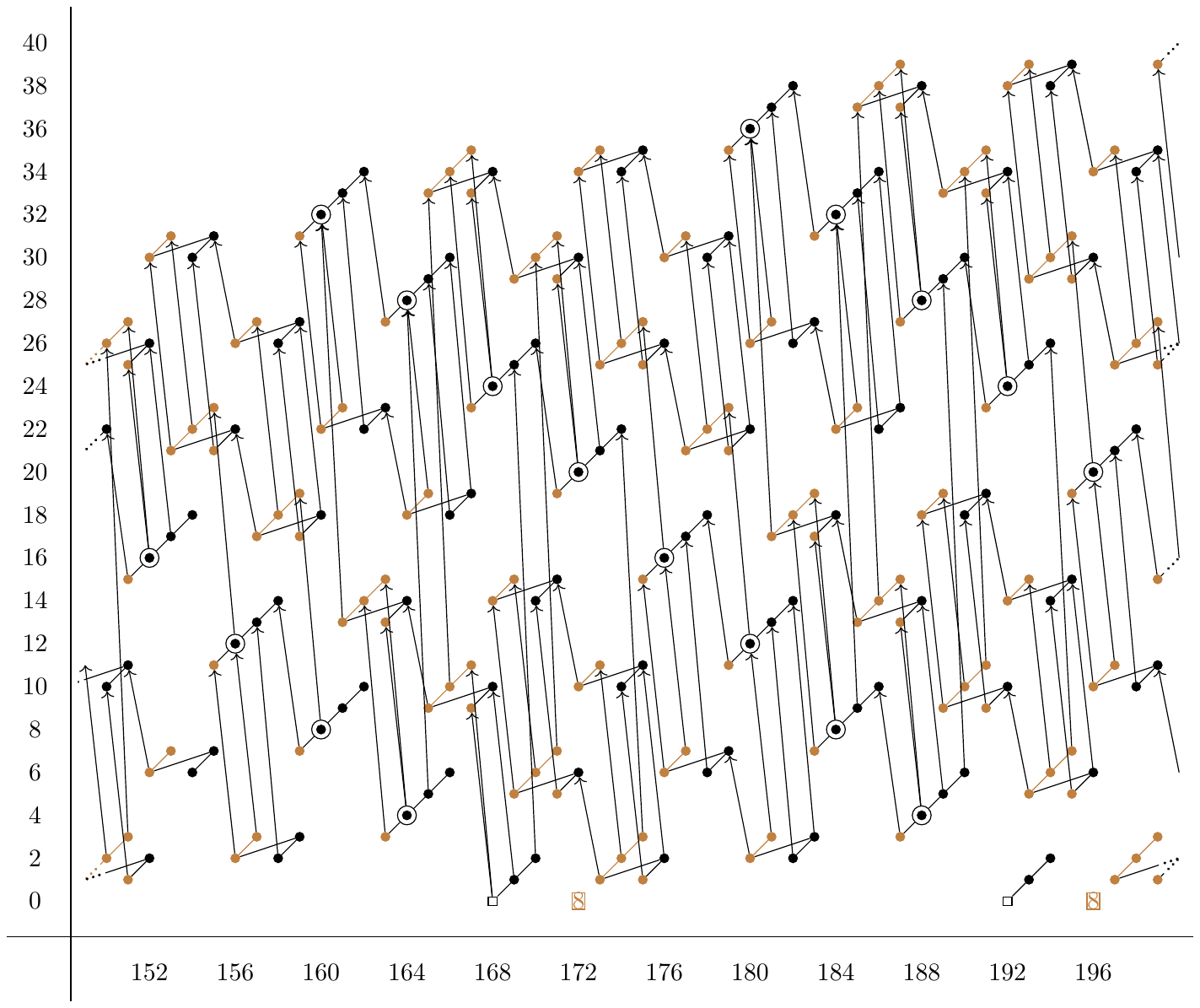}
    \caption{The higher pages of DSS for $\TJF_{2}$ and differentials}
    \label{fig:TJF2higher4}
\end{figure}

\FloatBarrier

\subsection{\texorpdfstring{$\TJF_3$}{TJF3}}
Diagrams for $\TJF_3$ are given below.

\par 
\begingroup
\centering

\vspace*{\fill} 

\widegraphic[page=1]{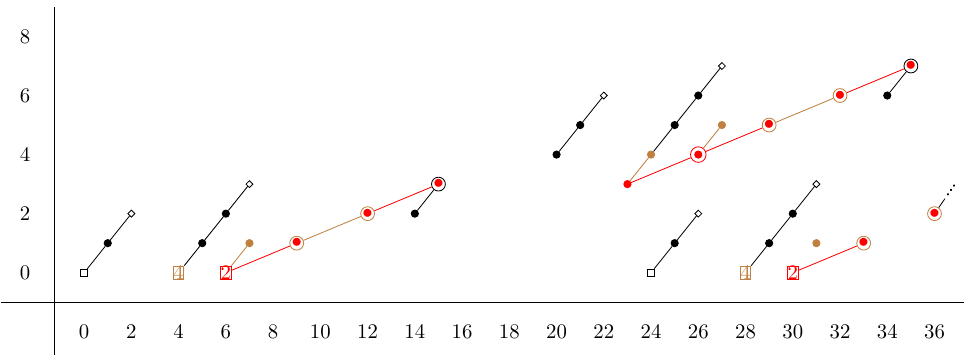}\par
\captionof{figure}{$E_2$-page of DSS for $\TJF_{3}$}
\label{fig:TJF3E2}

\vfill 

\widegraphic[page=2]{TJF3sseq.pdf}\par
\captionof{figure}{$E_4$-page and $d_r$, $r\ge 5$ of DSS for $\TJF_{3}$}
\label{fig:TJF3higher1}

\vspace*{\fill}
\par
\endgroup

\begin{figure}[ht]
\centering
    \verticalgraphic[page=3]{TJF3sseq.pdf}
    \caption{$E_4$-page and $d_r$, $r\geq 5$ of DSS for $\TJF_{3}$}
    \label{fig:TJF3higher2}
\end{figure}
\begin{figure}[tp]
\centering
    \verticalgraphic[page=4]{TJF3sseq.pdf}
    \caption{$E_4$-page and $d_r$, $r\geq 5$ of DSS for $\TJF_{3}$}
    \label{fig:TJF3higher3}
\end{figure}
\begin{figure}[tp]
\centering
    \verticalgraphic[page=5]{TJF3sseq.pdf}
    \caption{$E_4$-page and $d_r$, $r\geq 5$ of DSS for $\TJF_{3}$}
    \label{fig:TJF3higher4}
\end{figure}

\FloatBarrier

\subsection{\texorpdfstring{$\TJF_4$}{TJF4}}
This section collects diagrams of the descent spectral sequence for $\TJF_4$.

\par 
\begingroup
\centering

\vspace*{\fill} 

\widegraphic[page=2]{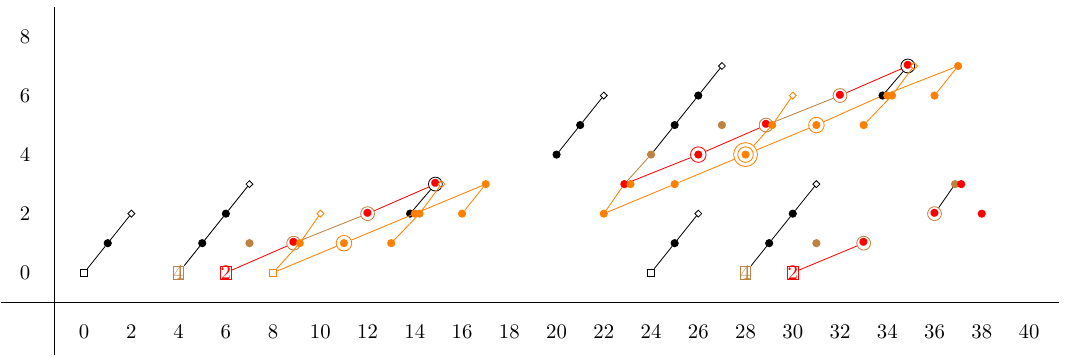}\par
\captionof{figure}{$E_2$-page of DSS for $\TJF_{4}$}
\label{fig:TJF4E2}

\vfill 

\widegraphic[page=3]{TJF4sseq.pdf}\par
\captionof{figure}{$E_4$-page and $d_r$, $r\ge 5$ of DSS for $\TJF_{4}$}
\label{fig:TJF4higher1}

\vspace*{\fill}
\par
\endgroup

\begin{figure}[ht]
\centering
    \verticalgraphic[page=4]{TJF4sseq.pdf}
    \caption{$E_4$-page and $d_r$, $r\geq 5$ of DSS for $\TJF_{4}$}
    \label{fig:TJF4higher2}
\end{figure}
\begin{figure}[tp]
\centering
    \verticalgraphic[page=5]{TJF4sseq.pdf}
    \caption{$E_4$-page and $d_r$, $r\geq 5$ of DSS for $\TJF_{4}$}
    \label{fig:TJF4higher3}
\end{figure}
\begin{figure}[tp]
\centering
    \verticalgraphic[page=6]{TJF4sseq.pdf}
    \caption{$E_4$-page and $d_r$, $r\geq 5$ of DSS for $\TJF_{4}$}
    \label{fig:TJF4higher4}
\end{figure}
\clearpage
\FloatBarrier

\subsection{\texorpdfstring{$\TJF_5$}{TJF5}}
This section collects diagrams of the descent spectral sequence for $\TJF_5$.

\par 
\begingroup
\centering

\vspace*{\fill} 

\widegraphic[page=2]{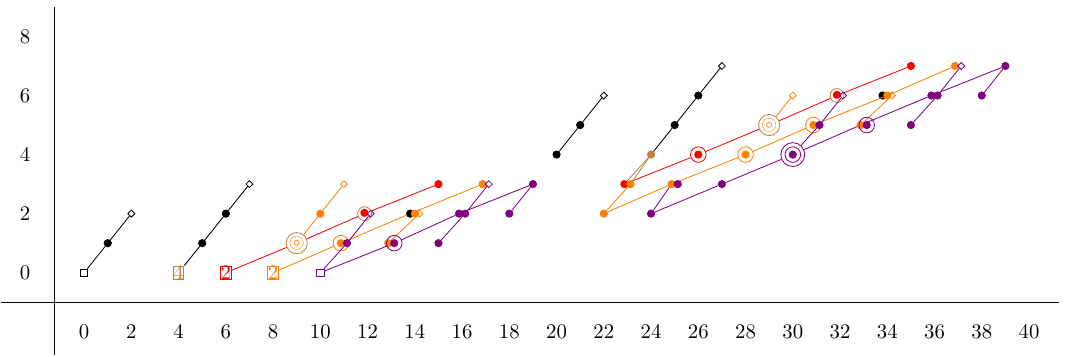}\par
\captionof{figure}{$E_2$-page of DSS for $\TJF_{5}$}
\label{fig:TJF5E2}

\vfill 

\widegraphic[page=3]{TJF5sseq.pdf}\par
\captionof{figure}{$E_4$-page and $d_r$, $r\ge 5$ of DSS for $\TJF_{5}$}
\label{fig:TJF5higher1}

\vspace*{\fill}
\par
\endgroup

\begin{figure}[ht]
\centering
    \verticalgraphic[page=4]{TJF5sseq.pdf}
    \caption{$E_4$-page and $d_r$, $r\geq 5$ of DSS for $\TJF_{5}$}
    \label{fig:TJF5higher2}
\end{figure}
\begin{figure}[tp]
\centering
    \verticalgraphic[page=5]{TJF5sseq.pdf}
    \caption{$E_4$-page and $d_r$, $r\geq 5$ of DSS for $\TJF_{5}$}
    \label{fig:TJF5higher3}
\end{figure}
\begin{figure}[tp]
\centering
    \verticalgraphic[page=6]{TJF5sseq.pdf}
    \caption{$E_4$-page and $d_r$, $r\geq 5$ of DSS for $\TJF_{5}$}
    \label{fig:TJF5higher4}
\end{figure}
\clearpage
\FloatBarrier

\subsection{\texorpdfstring{$\TJF_6$}{TJF6}}
This section collects diagrams of the descent spectral sequence for $\TJF_6$.

\par 
\begingroup
\centering

\vspace*{\fill} 

\widegraphic[page=2]{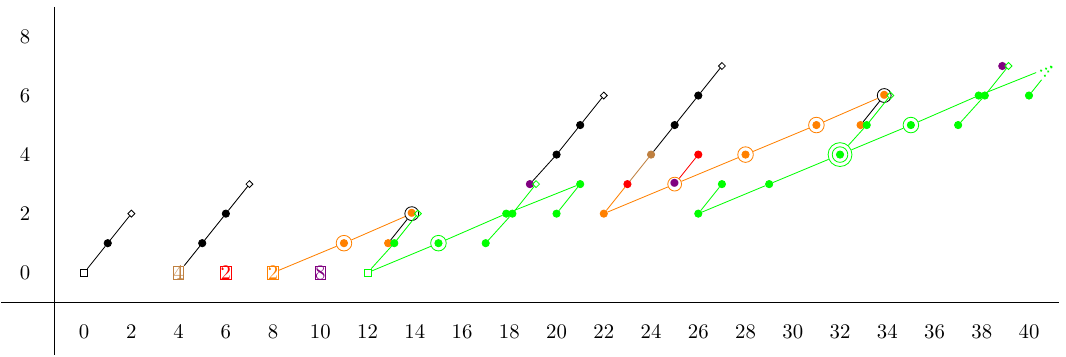}\par
\captionof{figure}{$E_2$-page of DSS for $\TJF_{6}$}
\label{fig:TJF6E2}

\vfill 

\widegraphic[page=3]{TJF6sseq.pdf}\par
\captionof{figure}{$E_4$-page and $d_r$, $r\ge 5$ of DSS for $\TJF_{6}$}
\label{fig:TJF6higher1}

\vspace*{\fill}
\par
\endgroup

\begin{figure}[ht]
\centering
    \verticalgraphic[page=4]{TJF6sseq.pdf}
    \caption{$E_4$-page and $d_r$, $r\geq 5$ of DSS for $\TJF_{6}$}
    \label{fig:TJF6higher2}
\end{figure}
\begin{figure}[tp]
\centering
    \verticalgraphic[page=5]{TJF6sseq.pdf}
    \caption{$E_4$-page and $d_r$, $r\geq 5$ of DSS for $\TJF_{6}$}
    \label{fig:TJF6higher3}
\end{figure}
\begin{figure}[tp]
\centering
    \verticalgraphic[page=6]{TJF6sseq.pdf}
    \caption{$E_4$-page and $d_r$, $r\geq 5$ of DSS for $\TJF_{6}$}
    \label{fig:TJF6higher4}
\end{figure}
\clearpage
\FloatBarrier

\subsection{\texorpdfstring{$\TJF_7$}{TJF7}}
This section collects diagrams of the descent spectral sequence for $\TJF_7$.

\par 
\begingroup
\centering

\vspace*{\fill} 

\widegraphic[page=2]{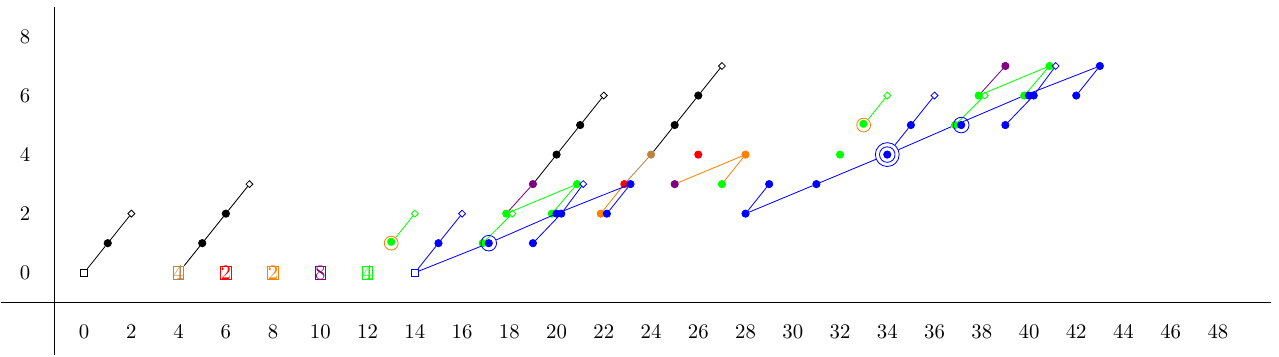}\par
\captionof{figure}{$E_2$-page of DSS for $\TJF_{7}$}
\label{fig:TJF7E2}

\vfill 

\widegraphic[page=3]{TJF7sseq.pdf}\par
\captionof{figure}{$E_4$-page and $d_r$, $r\ge 5$ of DSS for $\TJF_{7}$}
\label{fig:TJF7higher1}

\vspace*{\fill}
\par
\endgroup

\begin{figure}[ht]
\centering
    \verticalgraphic[page=4]{TJF7sseq.pdf}
    \caption{$E_4$-page and $d_r$, $r\geq 5$ of DSS for $\TJF_{7}$}
    \label{fig:TJF7higher2}
\end{figure}
\begin{figure}[tp]
\centering
    \verticalgraphic[page=5]{TJF7sseq.pdf}
    \caption{$E_4$-page and $d_r$, $r\geq 5$ of DSS for $\TJF_{7}$}
    \label{fig:TJF7higher3}
\end{figure}
\begin{figure}[tp]
\centering
    \verticalgraphic[page=6]{TJF7sseq.pdf}
    \caption{$E_4$-page and $d_r$, $r\geq 5$ of DSS for $\TJF_{7}$}
    \label{fig:TJF7higher4}
\end{figure}
\clearpage

\def\arxivconsists of font{\rm}
\bibliographystyle{ytamsalpha}
\bibliography{ref}

\end{document}